\documentclass{amsart}
 \usepackage{amsthm,amsfonts,amsmath,amssymb,latexsym,epsfig,upref,eucal,ae}

\usepackage[all]{xy}
\usepackage{color}

\theoremstyle{plain}
\newtheorem{theorem}{Theorem}[section]
\newtheorem{lemma}[theorem]{Lemma}
\newtheorem{corollary}[theorem]{Corollary}
\newtheorem{proposition}[theorem]{Proposition}

\theoremstyle{definition}
\newtheorem{definition}[theorem]{Definition}
\newtheorem{definition-theorem}[theorem]{Definition-Theorem}

\theoremstyle{remark}
\newtheorem{remark}[theorem]{Remark}

\def\Aut{\mathrm{Aut}}

\def\vol{\mathrm{vol}}

\def\Hilb{\mathrm{Hilb}}
\def\FS{\mathrm{FS}}

\def\Lie{\mathrm{Lie}}

\def\cF{\mathcal{F}}
\def\cO{\mathcal{O}}

\def\cG{\mathcal{G}}
\def\cH{\mathcal{H}}

\def\cB{\mathcal{B}}

\def\del{\partial}
\def\delb{\overline\partial}

\def\gl{\mathfrak{gl}}

\def\su{\mathfrak{su}}

\def\R{\mathbb{R}}

\def\N{\mathbb{N}}

\def\P{\mathbb{P}}

\def\L{\mathbb{L}}

\def\C{\mathbb{C}}

\def\bs{{\bf s}}
\def\uS{\underline{S}}

\def\Om{\Omega}
\def\om{\omega}
\def\ep{\varepsilon}

\def\>{\rangle}

\def\<{\langle}
\def\>{\rangle}

\begin{document}

\title[Moment map for balanced metrics on extremal K\"ahler manifolds]
{A moment map picture of relative balanced metrics on extremal K\"ahler manifolds}
\author[Y. Sano]{Yuji Sano}
\author[C. Tipler]{Carl Tipler}
\address{Department of Applied Mathematics
    Fukuoka University
    8-19-1 Nanakuma, Jonan-ku, Fukuoka 814-0180, JAPAN;
D\'epartement de math\'ematiques,
Universit\'e de Bretagne Occidentale, 6, avenue Victor Le Gorgeu, 29238 Brest Cedex 3 France}
\email{sanoyuji@fukuoka-u.ac.jp ; carl.tipler@univ-brest.fr}

\date{\today}

\begin{abstract}
We give a moment map interpretation of some relatively balanced metrics.
As an application, we extend a result of S.\,K.\,Donaldson on constant scalar curvature K\"ahler metrics to the case of extremal metrics. Namely, we show
that a given extremal metric is the limit of some specific relatively balanced metrics. As a corollary, we recover uniqueness and splitting results for extremal metrics in the polarized case.
\end{abstract}

\maketitle

\section{Introduction}
In \cite{Don01}, Donaldson gave a general framework to study some specific Fubini-Study metrics called \textit{balanced} metrics on a polarized manifold.
It is a finite dimensional counterpart of the moment map interpretation of constant scalar curvature K\"ahler (cscK, for short) metrics by Fujiki \cite{fujiki} and Donaldson \cite{Don97}.
Donaldson proved that a given cscK metric is the limit of balanced metrics if the automorphism group of the polarized manifold is discrete.
In this paper, we extend this framework and its applications to the case of extremal metrics by using some relatively balanced metrics introduced in the authors' previous work \cite{st}.

Let $(X, \omega)$ be an $n$-dimensional K\"ahler manifold.
A K\"ahler metric is called \textit{extremal} in the sense of Calabi \cite{c1} if and only if it is a critical point of the functional
$$
	\omega \mapsto  \int_{X} (S(\omega) - \uS)^{2} d\mu_{ \omega}
$$
defined over the space of K\"ahler metrics in a given K\"ahler class, where $S( \omega)$ is the scalar curvature of $\omega$, $d\mu_{ \omega}$ is the volume form $\omega^{n}/n!$ with respect to $\omega$ and $\uS$
is the average of the scalar curvature.
These metrics are natural generalizations of cscK metrics in the presence of holomorphic vector fields.

From now on, we consider the case where  $(X,L)$ is a polarized manifold, i.e., $L$ is an ample line bundle on $X$.
For an Hermitian metric $h$ on $L$, let us denote $-i \partial\bar{\partial}\log h$ by  $\omega_{h}$.
Then, the metric $h$ induces an inner product $\|\cdot\|_{\Hilb_{k}(h)}$ on $V_{k}=H^{0}(X, L^{\otimes k})$ defined by
$$
	\vert\vert s\vert\vert ^2_{\Hilb_k(h)}=\int_X \vert s \vert_{h^k}^2 d\mu_h,
$$
where $d\mu_{h}$ is the volume form with respect to $\omega_{h}$.
Taking an orthonormal basis $\mathbf{s}= \{s_{\alpha}\}_{\alpha=1}^{N_{k}}$ of $V_{k}$ with respect to $\Hilb_k(h)$,  $X$ can be embedded into $ \C\P^{N_k-1}$ for $k$ large enough.
An Hermitian metric $h$ (or its associated K\"ahler form $\omega_{h}$) is called  \textit{ $k^{th}$ balanced} if and only if the pulled-back Fubini-Study metric
$$
	\omega_{\FS_{k} \circ \Hilb_{k}(h)}=\frac{1}{k}i \partial\bar{\partial} \log \bigg(\frac{1}{N_k}  \sum_{\alpha} \vert s_{\alpha}\vert^2\bigg)
	\in 2 \pi c_{1}(L)
$$
is equal to $\omega_{h}$.
In \cite{Don01}, Donaldson proved that if $\Aut(X,L)$ is discrete,
and if $(X,L)$ admits a cscK metric $\omega_{csc} \in 2\pi c_{1}(L)$, then there exists for each $k \gg 0$
a unique $k^{th}$ balanced metric $\omega(k) \in 2\pi c_{1}(L)$. Moreover, the sequence $(\omega(k))_{k\gg0}$ converges to $\omega_{csc}$ in $C^{ \infty}$-sense.

Let us drop the discreteness assumption.
An extremal metric can be seen as a self-similar solution to the Calabi flow:
$$
	\frac{\partial \varphi_{t}}{\partial t} = S (\omega_{t})- \uS, \,\,
	\omega_{t}= \omega_{0} + i \partial\bar{\partial} \varphi_{t}.
$$
The quantization of this flow is an iterative process, called $T_{k}$-iteration \cite{Don05}:
$$
	T_k : \omega \mapsto \omega_{\FS_{k}\circ \Hilb_{k}(\omega)},
$$
whose fixed points are balanced metrics. We then introduce self-similar solutions for the $T_k$-iteration process.
Replacing $L$ by a sufficiently large tensor power if necessary, we have a group representation
$$
	\rho_{k}: \mathrm{Aut}_{0}(X,L) \to  \mathrm{SL}(V_{k}).
$$
Fix a maximal torus $T$ in $\mathrm{Aut}_{0}(X,L)$, its complexification $T^{c}$, and denote the image of $T^{c}$ under $\rho_{k}$ by $T^{c}_{k}$.
As introduced in \cite{st}, we call $h$ (or $\omega_{h}$) \textit{$k^{th}$ $\sigma_{k}$-balanced } if and only if there exists $\sigma_{k} \in T^{c}_{k}$ such that
\begin{equation}\label{eq:sigmabalanceddefinition}
	\omega_{\FS_{k} \circ \Hilb_{k}(h)} = \sigma_{k}^{*}(\omega_{h}).
\end{equation}
It is a natural specific example of relative balanced metric as discussed in \cite{ma040} (see Remark \ref{rem:sigma balanced is relative balanced}).
Then, the main result of this paper is as follows:
\begin{theorem}\label{thm:A}
 Let $(X,L)$ be a polarized K\"ahler manifold with $\om_{ex}\in 2\pi c_1(L)$ extremal and let $T$ be the identity component of the isometry group of $\omega_{ex}$.
 Then there exists $k_0\in \N$ such that for all
 $k\geq k_0$, $(X,L^k)$ admits a $k^{th}$ $\sigma_{k}$-balanced metric $\omega(k)$ for some $\sigma_{k} \in T^{c}_{k}$. These metrics $(\omega(k))_{k\geq k_0}$ converge to the initial metric $\om_{ex}$ in $C^{ \infty}$-sense.
\end{theorem}
\noindent In order to prove Theorem \ref{thm:A}, we provide a moment map interpretation of $\sigma$-balanced metrics in Section \ref{sec:mmap}.
We also provide a characterization of the optimal weights $\sigma_k$ used in Theorem \ref{thm:A} in terms of characters on the Lie algebra of $T^c$.
\begin{remark}
 The choice of the optimal weights $\sigma_k$ determines a quantization of the extremal vector field, see Section \ref{sec:optimalweight}.
\end{remark}
\noindent Theorem \ref{thm:A} has the following two applications.
First, it simplifies Mabuchi's proof of the uniqueness of extremal metrics on polarized manifolds \cite{ma04}.
Indeed, Theorem \ref{thm:A} allows us to apply directly Kempf-Ness theorem from the theory of moment maps.
\begin{corollary}
\label{cor:b}
 Let $(X,L)$ be a polarized K\"ahler manifold. An extremal K\"ahler metric in $2 \pi c_1(L)$, if it exists, is unique up to automorphisms
 of $(X,L)$.
\end{corollary}
\noindent
Second, it provides a generalization of the splitting theorem of Apostolov-Huang \cite{ah}.
\begin{corollary}\label{cor:c}
 Let $(X=X_1\times X_2, L= L_1\otimes L_2)$ be a product of polarized K\"ahler manifold. Assume that $X$ admits an
 extremal K\"ahler metric $g$ in the class $ 2\pi c_1(L)$. Then $g$ is a product metric $g_1\times g_2$, where $g_i$ is an extremal metric
 on $X_i$ in the class $ 2\pi c_1(L_i)$.
\end{corollary}
\noindent This theorem is proved in \cite{ah} with stronger assumptions.
In particular, Theorem \ref{thm:A} was conjectured in \cite{ah} to obtain full generality of the above splitting theorem.

We finish this introduction with a brief review on relevant works to Theorem \ref{thm:A} (see also \cite{ah, hashimoto} for comprehensive reviews).
The approximation of canonical K\"ahler metrics by specific Fubini-Study metrics is closely related to the stability of $(X,L)$ in the sense of Geometric Invariant Theory (GIT).
In fact, it is well known (\cite{zha, luo, PS03, paul, wa}) that the existence of a balanced metrics on $(X,L^{\otimes k})$
is equivalent to the Chow stability of the embedding of $X$ to the projective space by sections of $L^{\otimes k}$ if $\mathrm{Aut}(X,L)$ is discrete.
The result in \cite{Don01} implies that if a polarized manifold admits a cscK metric, under the discreteness assumption, then $(X,L)$ is asymptotically Chow stable.
This is one of the early evidences for the so-called {\it Yau-Tian-Donaldson conjecture} which states
that the existence of canonical K\"ahler metrics on a polarized manifold should be equivalent to some stability notion of the manifold in the sense of GIT.
Extensions of \cite{Don01} to the case where $\mathrm{Aut}(X,L)$ is not discrete has been pioneered by Mabuchi \cite{ma040, ma04, ma04-1, ma12, ma16}.
Without the discreteness of $\mathrm{Aut}(X,L)$, even if $(X,L)$ admits cscK metrics, we cannot expect the existence of balanced metrics on $(X, L^{\otimes k}) $ for $k \gg 0$.
Counter-examples, i.e. asymptotic Chow unstable manifolds with cscK metrics, are found (\cite{osy, dz}).
\begin{remark}\label{rem:OSY-DVZ}
Theorem \ref{thm:A} says that on the examples in \cite{osy, dz} a cscK metric
can be approximated by non-trivial $\sigma_{k}$-balanced metrics. In particular, the vector fields induced by $\sigma_{k}$ will converge to zero (see Proposition \ref{prop:weight}).
\end{remark}
In fact, there may exist $v \in \Lie (T_{k}^{c})$ inducing a non-trivial action on the line where the Chow form of $(X, L^{\otimes k}) $ lies.
This action violates the Chow semistability of $(X, L^{\otimes k})$.
To avoid this phenomenon, in \cite{ma04-1, ma05}, some additional condition is required.
Such condition is reformulatd as the vanishing of the family of integral invariants, so-called higher Futaki invariants, by Futaki \cite{futaki04}.
In fact, the counter-examples stated in Remark \ref{rem:OSY-DVZ} are given by proving the non-vanishing of the higher Futaki invariants.
However, considering extremal metrics, the above requirement cannot be satisfied, because the action induced by the (non-trivial) extremal vector field violates it.
Studying the extension of GIT to the non-discrete case, Mabuchi introduced balanced metrics relative to a given torus in the identity component of the automorphism group of $X$ in \cite{ma040}.
Then, in \cite{ma09}, he proved that the existence of extremal metrics implies the asymptotic existence of relative balanced metrics.
A difference between \cite{ma09} and Theorem \ref{thm:A} is the choice of the group action on $V_{k}$.
The group considered in \cite{ma09} is $\Pi_\chi SU(N_k^\chi)$, smaller than (\ref{eq:Gk}), considered in the present work .
This difference affects the choice of the weight $\{\lambda_{j}\}$ of relative balanced metrics in (\ref{eq:twistedbergman}).
In particular, in \cite{ma09} it is not sure that the weight comes from a torus action.
This lack of information prevents one to apply Szekelyhidi's generalization of Kempf-Ness theorem \cite{sz}.
This is a source of difficulties in applications of \cite{ma09} to other related problems on extremal metrics.
For instance, delicate work is necessary to prove the uniqueness of extremal metrics on polarized manifolds in \cite{ma04}.
Hence, a refinement of the results in \cite{ma09} was expected, e.g. \cite{ah}.
Very recently, results equivalent to Theorem \ref{thm:A} are proved by Seyyedali \cite{seyyedali} and Mabuchi \cite{ma16} independently. 
Hashimoto also gives another quantization of extremal metrics \cite{hashimoto}.
Let us explain differences between \cite{seyyedali, ma16} and the proof of Theorem \ref{thm:A}.
While the weight of relative balanced metrics comes from a given extremal metric in \cite{seyyedali, ma16}, we prove that the weight of $\sigma$-balanced metrics is determined apriori regardless of the existence of extremal metrics.
The latter is a quite natural statement, because the weight of relative balanced metrics approximates the extremal vector field, which exists regardless of the existence of extremal metrics.
A motivational observation for our proof is that a $\sigma$-balanced metrics is self-similar for $T_{k}$-iteration.
Seeing $T_{k}$-iteration as a quantization of the Calabi flow the above observation corresponds to the fact that an extremal metric is a self-similar solution to the Calabi flow.
With this point of view, we use an argument analogous to one coming from the theory of K\"ahler-Ricci solitons \cite{tz}.
Our strategy is as follows.
First, we twist the moment map in \cite{Don01} by a given $\sigma$ (Section \ref{subsec:momentmap}).
By general theory, it induces a new invariant which is a generalization of the integral invariant considered in \cite{ma04-1}.
Then we can find the optimal $\sigma$ so that this new invariant vanishes (Proposition \ref{prop:futakimorita}).
Then, the obstruction considered in \cite{ma04-1} will vanish in our twisted setting, and we can adapt the arguments in \cite{Don01} and \cite{ps04} (Section \ref{sec:proof}).

\subsection{Plan of the paper}
In Section \ref{sec:setup}, we collect necessary definitions.
In Section \ref{sec:mmap}, we give a moment map interpretation for $\sigma$-balanced metrics.
In Section \ref{sec:optimalweight}, we choose the optimal weight $\sigma_{k}$ for each $k \gg 0$.
In Section \ref{sec:proof}, we complete the proof of Theorem \ref{thm:A}, following \cite{Don01,ps04}.

\subsection{Acknowledgments}
The authors would like to thank Vestislav Apostolov, Hugues Auvray, Yoshinori Hashimoto and Julien Keller for stimulating discussions and useful remarks.
YS is supported by MEXT, Grant-in-Aid for Young Scientists (B), No. 25800050.
CT is partially supported by ANR project EMARKS No ANR-14-CE25-0010.

\section{Setup}
\label{sec:setup}
In this section, we introduce some necessary material and results that will be used throughout the paper.
Let $(X,L)$ be a polarized K\"ahler manifold of complex dimension $n$. Let $\cH$ be the space of smooth K\"ahler potentials with respect to a fixed K\"ahler form $\om \in 2\pi c_1(L)$  :
\begin{eqnarray*}
\cH= \lbrace \phi \in C^{\infty} (X) \;\vert \; \om_{\phi}:=\om + i\del\delb \phi > 0 \rbrace.
\end{eqnarray*}

\subsection{Extremal vector field}
We fix $T$ a maximal torus of $\Aut_0(X,L)$.
By a theorem of Calabi \cite{c2}, the quest for extremal metrics can be done modulo the $T$-action.
We define $\cH^T$ to be the space of $T$-invariant potentials with respect to a $T$-invariant base point $\om$.
We say that a vector field $v$ is a Hamiltonian vector field
if there is a real valued function $\theta_v$, such that $\om(v,\cdot)= -d\theta_v$. If in addition $v$ is Killing, $\theta_v$ will be called a Killing potential.
\begin{remark}
Recall that for any $v\in \Lie(T^c)$
there exists a smooth function $\theta_{v, \omega}$ such that
\begin{equation}\label{eq:potentialfunctiontheta}
	\iota_{v}\omega = - \bar{\partial} \theta_{v, \omega}.
\end{equation}
We normalize $\theta_{v, \omega}$ by
\begin{equation}\label{eq:normalizationthetasigma}
	\int_{X} \theta_{v, \omega} d \mu_{ \omega}=0.
\end{equation}
Then, the infinitesimal action of $v$ to $L$ is given by
$$
	v^{\sharp}=-2\pi i \theta_{v, \omega} z \frac{\partial}{\partial z} + v^{h}
$$
where $z$ is the fiber coordinate on $L$ and $v^{h}$ is the horizontal lift with respect to the connection with curvature $\omega$.
The normalization (\ref{eq:normalizationthetasigma}) determines uniquely the lift of the $v$-action on $X$ to $L$ (see Remark 2.2 in \cite{fm2}).
\end{remark}
\noindent For any $\phi\in\cH^T$, let $P_{\phi}^T$ be the space of normalized Killing potentials with respect to $\om_{\phi}$
whose corresponding Hamiltonian vector fields lie in $\Lie (T)$. Let $\Pi_{\phi}^T$ be the orthogonal projection from $L^2(X,\R)$ to $P_{\phi}^T$ given by the inner product on functions
\begin{equation}\label{eq:bilinearfutakimabuchi}
(f,g) \mapsto \int_X fg d\mu_{\phi}.
\end{equation}
Note that $T$-invariant metrics $\om_\phi$ satisfying $S(\phi)=\uS+\Pi_{\phi}^T S(\phi)$ are extremal.
\begin{definition}
\label{def:extremalvectorfield}
The extremal vector field $v_{ex}^T$ (or $v_{ex}$) with respect to $T$ is defined by the following formula, for any $\phi\in\cH^T$:
$$
v_{ex}^T=\nabla_g (\Pi_\phi^T S(\phi)).
$$
By \cite[Proposition 4.13.1]{gbook}, the extremal vector field does not depend on $\phi\in \cH^T$.
\end{definition}

\subsection{Quantization}
For each $k$, we can consider $\cH_k$ the space of hermitian metrics on $L^{\otimes k}$. To each element $h\in \cH_k$ one associates a K\"ahler metric $ -i \del\delb \log h$ on $X$, identifying the spaces $\cH_k$ to $\cH$. Write $\om_h$ to be the curvature of the hermitian metric $h$ on $L$. Fixing a base metric $h_0$  in $\cH_1$ such that $\om=\om_{h_0}$ the correspondence reads
\begin{eqnarray*}
\om_{\phi}=\om_{e^{-\phi}h_0}=\om+i\del\delb \phi .
\end{eqnarray*}
We denote by $\cB_k$ the space of positive definite Hermitian forms on $V_{k}:=H^0(X,L^{\otimes k})$. Let $N_k= \dim V_{k}$.
The spaces $\cB_k$ are identified with $GL_{N_k}(\C)/ U(N_k)$ using the base metric $h_0^k$. These symmetric spaces come with metrics $d_k$ defined by  Riemannian metrics:
\begin{eqnarray*}
(H_1,H_2)_h=Tr(H_1H^{-1}\cdot H_2 H^{-1}).
\end{eqnarray*}
There are maps :
\begin{eqnarray*}
\Hilb_k :  \cH & \rightarrow &\cB_k \\
\FS_k : \cB_k &\rightarrow &\cH
\end{eqnarray*}
defined by :
\begin{eqnarray*}
\forall h\in \cH\;, \; s\in V_k\;, \; \vert\vert s\vert\vert ^2_{\Hilb_k(h)}=\int_X \vert s \vert_{h^k}^2 d\mu_h
\end{eqnarray*}
and
\begin{eqnarray*}
\forall H \in \cB_k\; , \;
\FS_k(H)= \frac{1}{k} \log \bigg(\frac{1}{N_k}  \sum_{\alpha} \vert s_{\alpha}\vert_{h_0^k}^2\bigg)
\end{eqnarray*}
where  $\bs=\lbrace s_{\alpha}\rbrace$ is an orthonormal basis of $V_k$ with respect to $H$.
For any $\phi\in\cH$ and $k>0$, let $\lbrace s_{\alpha} \rbrace$ be an orthonormal basis of $V_k$ with respect to $\Hilb_k(\phi)$. The $k^{th}$ Bergman function of $\phi$ is defined to be :
$$
\rho_k(\phi)=\sum_{\alpha}\vert s_{\alpha}\vert^2_{h^k}.
$$
It is well known that a metric $\phi\in \Hilb_k(\cH)$ is balanced if and only if $\rho_k(\phi)$ is constant.
A key result in the study of balanced metrics is the following expansion:
\begin{theorem}[\cite{cat},\cite{ruan},\cite{tian90},\cite{zel}]
The following uniform expansion holds
$$
\rho_k(\phi)=k^n+A_1(\phi)k^{n-1}+A_2(\phi)k^{n-2}+...
$$
with $A_1(\phi)=\frac{1}{2}S(\phi)$ is half of the scalar curvature of the K\"ahler metric $\om_\phi$
and for any $l$ and $R\in \N$, there is a constant $C_{l,R}$ such that
$$
\vert\vert\rho_k(\phi) -\sum_{j\leq R} A_j k^{n-j} \vert \vert_{C^l} \leq C_{l,R}k^{n-R}.
$$
\end{theorem}
\noindent In particular we have the convergence of metrics
\begin{equation}
\label{cor:ber}
\om_{\FS_k\circ\Hilb_k(\phi)}=\om_{\phi}+\cO(k^{-2}).
\end{equation}
By integration over $X$ we also deduce
\begin{equation}
\label{cor:Nk}
N_k=k^n\mathrm{Vol}(X)+\frac{1}{2} \mathrm{Vol}(X) \uS k^{n-1}+\cO(k^{n-2}).
\end{equation}
where
$$
\uS=2n\pi\frac{c_1(-K_X)\cup[\om]^{n-1}}{[\om]^n}
$$
is the average of the scalar curvature and $\mathrm{Vol}(X)$ is the volume of $(X,c_1(L))$.

\section{A moment map interpretation of $\sigma$-balanced metrics}
\label{sec:mmap}
In this section, we provide a moment map description for $\sigma$-balanced metrics.
We closely follow the treatment in \cite{gbook}
(see also \cite{ah}).
\begin{definition}
We call $\phi\in\cH$ a \textit{$k^{th}$ $\sigma_{k}$-balanced metric} if there exists $\sigma_{k} \in T^{c}_{k}$ such that
\begin{equation}
\label{eq:sigmabalanceddefinitionbis}
	\omega_{k\FS_{k} \circ \Hilb_{k}(\phi)} = \sigma_{k}^{*}(\omega_{k\phi}).
\end{equation}
\end{definition}
\begin{remark}
\label{rem:sigma balanced is relative balanced}
Taking an appropriate orthonormal basis $\bs$ of $V_{k}$ in which $\sigma_{k}$ is diagonal, equation (\ref{eq:sigmabalanceddefinitionbis}) is equivalent to the twisted Bergman function
\begin{equation}\label{eq:twistedbergman}
	\sum_{j=1}^{N_{k}} e^{ -\lambda_{j}}|s_{j}|^{2}_{h^{k}}
\end{equation}
being constant on $X$, where
$$
	\sigma_{k}= \exp(\frac{1}{2}\mathrm{diag}(\lambda_{1}, \ldots, \lambda_{N_{k}})), \,\, \lambda_{j} \in \R.
$$
This is also equivalent to the fact that the embedding of $X$ to $\C\P^{N_k-1}$ using $\bs$ satisfies
$$
	\frac{e^{- \frac{1}{2}( \lambda_{\alpha} + \lambda_{\beta})}}{\mathrm{Vol}(X)} \int_{X} \frac{ s_{\alpha} \overline{ s_{\beta}} }{\sum |s_{\gamma}|^{2}} d \mu_{\omega_{\FS_{k} \circ \Hilb_{k}(h)}}
	= \delta_{\alpha\overline{\beta}}.
$$
These characterizations tell us that a $\sigma$-balanced metric is a specific relative balanced metric as discussed in \cite{ma040}.
\end{remark}
\subsection{The relative setting}
We extend the quantization tools to the extremal metrics setup.
Replacing $L$ by a sufficiently large tensor power if necessary, we can assume that $\Aut_0(X,L)$ acts on $L$ (see e.g. \cite{kob}).
We then consider the group representation
$$
	\rho_{k}: \mathrm{Aut}_{0}(X,L) \to  \mathrm{SL}(V_{k}).
$$
Recall that $T$ is a maximal torus of $\Aut_0(X,L)$.
The $T$-action on $X$ induces a $T$-action on the space $\cB_k$ and we define $\cB_k^T$ to be the subspace of $T$-invariant elements.
The spaces $\cB_k^T$ are totally geodesic in $\cB_k$ for the distances $d_k$ and we have the induced maps :
\begin{equation}
\label{eq:Hilb}
\begin{array}{cccc}
\Hilb_k : & \cH^T & \rightarrow &\cB_k^T \\
\FS_k :& \cB_k^T &\rightarrow &\cH^T.
\end{array}
\end{equation}
The action of the complexified torus $T^{c}_{k}:=\rho_k(T^c)$ on
$V_k$ induces a weight decomposition
$$
V_k=\bigoplus_{\chi\in w_k(T)} V_{k}(\chi)
$$
where $\rho_k(T^c)$ acts on $V_{k}(\chi)$ with weight $\chi$, and $w_k(T)$ is the space of weights for this action.
Let $N_k^\chi$ be the dimension of $V_{k}(\chi)$.
We consider the space of basis:
$$
\cB^T(V_k):=\left \{ ( s_i^{\chi} )_{\chi \in w_k(T); i=1..N_k^\chi} \in (V_k)^{N_k} \vert \det(s_i^\chi)\neq 0 \text{ and } \forall (\chi, i),\; s_i^\chi\in V_{k}(\chi)  \right \}.
$$
For each $k$, we define the subgroup
$$\Aut_k^T \subset SL(V_k)$$
to be the centralizer of $T^{c}_{k}$ in
$\rho_k(\Aut_0(X))$ and the space $Z^T(V_k)$ to be the quotient
$$
Z^T(V_k)=\cB^T(V_k) / (\C^* \times \Aut_k^T),
$$
where $\C^*$ acts by scalar multiplication.
Then consider the group
$$G^c_k=S(\Pi_\chi GL_{N_k^\chi}(\C))$$
which is the complexification of
\begin{equation}\label{eq:Gk}
	G_k:=S(\Pi_\chi U(N_k^\chi)).
\end{equation}
There is a natural right action of $G^c_k$ on $\cB^T(V_k)$ that commutes with the left action of $\C^* \times \Aut_k^T$ on
$\cB^T(V_k)$. Then the actions of these groups descend to actions on the quotient $Z^T(V_k)$.
We will see in the next section that the space $Z^T(V_k)$ carries a K\"ahler structure
such that the $\sigma_k$-balanced condition appears as the vanishing of a moment map with respect to the $G_k$-action.

\subsection{A K\"ahler structure on $Z^T(V_k)$ for weighted considerations}

In this section, we will abbreviate the subscript $k$ if it does not lead to confusion.
As a space of basis for a complex vector space, $\cB^T(V)$ carries a natural integrable almost-complex structure $J_{\cB}$ that descends to an integrable almost-complex structure $J_Z$ on the quotient $Z^T(V)$.
Then we build a symplectic form as follows.
First of all, to each $\bs\in \cB^T(V)$ we can associate a unique element $H(\bs)\in \cB^T$ so that
$\bs$ is an orthonormal basis of $H(\bs) $. Note also that there is a map:
$$
\begin{array}{cccc}
 \phi : & \cB^T(V) & \rightarrow & \cH^T \\
  & \bs & \mapsto & \FS(H(\bs)),
\end{array}
$$
We will sometimes write $\phi_\bs$ for $\phi(\bs)$.
\begin{remark}
At $\mathbf{s}=\lbrace s_{\alpha}\rbrace\in \cB^T(V_k)$, we define an isomorphism
\begin{equation}
 \label{eq:isomPhis}
\begin{array}{cccc}
 \Phi_\bs: & \P(V_k^*) &\rightarrow &\C\P^{N_k-1}\\
           &  [ \mathrm{ev}]  & \mapsto & [ \mathrm{ev}(s_\alpha)]
\end{array}
\end{equation}
If $\iota : X \hookrightarrow \P(V^*)$ denotes the Kodaira embedding,
then $\om_{\phi_\bs}= (\Phi_\bs \circ \iota)^*\om_{\FS}$.
\end{remark}
\noindent Fix an element $\sigma \in T^{c} \subset \mathrm{SL}(V)$.
We introduce a K\"ahler form on $\cB^T(V)$, twisted by $\sigma$, as follows.
Take $v\in \Lie(T^{c})$ so that $\exp(v)= \sigma$.
For a given metric $\omega_{\phi}$, define the function $\psi_{\sigma,\phi}$ by
\begin{equation}
\label{def:psi}
\sigma^*\om_{\phi}=\om_{\phi}+i \del\delb \psi_{\sigma,\phi}
\end{equation}
with the normalization
\begin{equation}
\label{eq:normalization_phi}
\int_X \exp{(\psi_{\sigma,\phi})}\; d\mu_{\phi} =\frac{N_{k}}{k^{n}}.
\end{equation}
Then we consider a modified Aubin functional introduced in \cite{st} defined up to a constant by its
differential:
$$
d I^{\sigma}(\phi)(\delta\phi)=\int_X \delta \phi (1+\Delta_{\phi})e^{\psi_{\sigma,\phi}} d\mu_{\phi}
$$
where $\Delta_{\phi}=-g_{\phi}^{i\overline{j}}\frac{\del}{\del z_i}\frac{\del}{\del \overline{z}_j}$
is the complex Laplacian of $g_{\phi}$.
Define the $2$-form $\Omega_\cB^\sigma$ on $\cB^T(V)$ by
$$
\Omega_\cB^\sigma(\bs):=dd^c(I^{\sigma}\circ\phi(\bs))
$$
where $d^c$ is defined with respect to $J_\cB$.
Then we prove the following:

\begin{proposition}
 \label{prop:symp}
 The $2$-form $\Omega_\cB^\sigma$ descends to a $G$-invariant symplectic form $\Om_Z^\sigma$ on $Z^T(V)$ such that $(Z^T(V), J_Z, \Om_Z^\sigma)$ is K\"ahler.
\end{proposition}

\begin{proof}
Let us denote by $\Theta_\cB^\sigma$ the $1$-form $d^c(I^{\sigma}(\phi))$. We show that $\Theta_\cB^\sigma$ is invariant under the actions
of $G$ and $\Aut^T$. Then $\Theta_\cB^\sigma$ and $\Omega_\cB^\sigma=d \Theta_\cB^\sigma$ descend to $G$-invariant forms on $Z^T(V)$.
By definition, for any $\bs\in\cB^T(V)$ and any $A\in T_\bs\cB^T(V)$,
$$
\Theta_\cB^\sigma(\bs)(A)=d^c(I^{\sigma}\circ\phi(\bs))(A)=-d(I^{\sigma}\circ\phi(\bs))(iA)=-dI^\sigma(\phi_\bs)(D_\bs\phi(iA))
$$
and thus
$$
\Theta_\cB^\sigma(\bs)(A)=\int_X D_\bs\phi(iA) (1+\Delta_{\phi_\bs})e^{\psi_{\sigma,\phi_\bs}} d\mu_{\phi_\bs}.
$$
Note that at each point $\bs\in\cB^T(V)$, the isomorphism $\Phi_\bs$ from (\ref{eq:isomPhis}) induces an isomorphism
$$
	T_\bs\cB^T(V)\simeq  \bigoplus_\chi \gl_{N^\chi}(\C).
$$
We then denote by $\hat{A}$ the vector field on $\C\P^N$
induced by $A= (A_{ij})\in T_\bs \cB^T(V)$.
Set $s=\lbrace s_j \rbrace$. A direct computation shows (e.g. \cite[lemma 8.4.3]{gbook})
$$
D_\bs\phi(iA)=-\dfrac{\sum_{ij} A^-_{ij}(s_i,s_j)_{h_0}}{\sum_k \vert s_k \vert_{h_0}^2}
$$
where $A^-$ is the anti-hermitian part of $A$. In other words, if $\mu^{\hat{A^-}}$ denotes the momentum of $\hat{A^-}$ on $\C\P^N$,
$$
D_\bs\phi(iA)=-\mu^{\hat{A^-}}\circ \Phi_s \circ \iota.
$$
To simplify notations, we let $\mu_\bs^{\hat{A^-}}=\mu^{\hat{A^-}}\circ \Phi_s \circ \iota$, so that
\begin{equation}
\label{eq:theta}
\Theta_\cB^\sigma(\bs)(A)=-\int_X \mu_\bs^{\hat{A^-}} (1+\Delta_{\phi_\bs})e^{\psi_{\sigma,\phi_\bs}} d\mu_{\phi_\bs}.
\end{equation}
Now from the definition of the action of $G$ on $\cB^T(V)$,
$$\forall g\in G,\; \om_{\phi(\bs\cdot g)}=\om_{\phi(\bs)}.$$
Thus $\Theta_\cB^\sigma$ is $G$-invariant. Then (\cite{gbook} Proposition 8.3.2)
$$
\forall \gamma \in \Aut^T,\; \om_{\phi(\gamma \cdot \bs)}=\gamma^*\om_{\phi(\bs)},
$$
and a change of variables in (\ref{eq:theta}) shows that $\Theta_\cB^\sigma$ is $\Aut^T$-invariant.

It remains to show that $g_\cB^\sigma:=\Omega_\cB^\sigma(\cdot,J_\cB\cdot)$ is positive and vanishes exactly on the distribution given by the leaves of the
$\C^*\times \Aut^T$-orbits. Let $\bs\in\cB^T(V)$ and $A\in T_\bs\cB^T(V)$. We have
$$
g_\cB^\sigma(A,A)=\Omega_\cB^\sigma(A,J_\cB A)=dd^c(I^{\sigma}\circ\phi)(A,J_\cB A).
$$
Let $\bs(t)=\bs \cdot\exp(tA)$ and
$\bs(t)^c=\bs \cdot\exp(tJ_\cB A)$. Then, as $J_\cB$ is integrable,
\begin{equation}
 \label{eq:gB}
g_\cB^\sigma(A,A)=\dfrac{d^2}{dt^2}\bigg\vert_{t=0} (I^{\sigma}\circ\phi) (\bs(t))+\dfrac{d^2}{dt^2}\bigg\vert_{t=0} (I^{\sigma}\circ\phi )(\bs(t)^c).
\end{equation}
\noindent
If $A\in T_\bs \cB^T(V)\simeq \bigoplus_\chi \gl_{N^\chi}(\C)$ is diagonal, that is to say, corresponds to the $\C^*$ action on $\cB^T(V)$,
then we easily compute $g_\cB^\sigma(A,A)=0$. Now if we assume $A$ to be trace-free,
by  \cite[Lemma 3.3.1]{st}
applied to
$$
	\phi(\bs(t))=\log(\sum_{\alpha} \vert s_\alpha\cdot \exp(tA) \vert_{h_0}^2)
$$
and
$$
	\phi(\bs(t)^c)=\log(\sum_{\alpha} \vert s_\alpha\cdot \exp(tJ_\cB A) \vert_{h_0}^2),
$$
we deduce that
$$
g_\cB^\sigma(A,A)\geq 0
$$
with equality if and only if $t\mapsto\exp(tA)$ and $t\mapsto\exp(t J_\cB A)$ (or more precisely the subgroups of $SL(V)$ determined
by $\bs(t)$ and $\bs(t)^c$) are in $\Aut^T$. This concludes the proof.
\end{proof}

\subsection{The moment map for weighted considerations}
\label{subsec:momentmap}
We define
$$
\begin{array}{cccc}
 \mu^\sigma : & \cB^T(V) & \rightarrow & \Lie(G) \\
  &  \bs=\lbrace s_j \rbrace & \mapsto & i\Hilb(\phi_\bs)_{0}( \sigma\cdot s_{j},\, \sigma\cdot s_{k})
\end{array}
$$
where the subscript $0$ means the trace-free part of the matrix.
The $\sigma$-balanced condition corresponds to the existence of a basis $\bs\in (\mu^\sigma)^{-1}(0)$:

\begin{proposition}
 \label{prop:sigmammap}
 For any $\bs\in \cB^T(V)$, $\om_{\phi(\bs)}$ is $\sigma$-balanced if and only if $\mu^\sigma(\bs)=0$.
\end{proposition}

\begin{proof}
Let $\bs=\lbrace s_j\rbrace\in\cB^T(V)$. By definition,
$\mu^\sigma(\bs)=0$ if and only if $c^{-1/2}\sigma^*\bs$ is an orthonormal basis for $\Hilb(\phi_\bs)$, for some constant $c \neq 0$. This is the same as
$\om_{\phi_\bs}$ being $\sigma$-balanced.
\end{proof}
\noindent
Lastly, we show that $\mu^\sigma$ is indeed a moment map for our setup:
\begin{proposition}
 \label{prop:mmap}
 The map $\mu^\sigma$ descends to an equivariant moment map for the $G$-action on $Z^T(V)$ with respect to $\Om_Z^\sigma$.
\end{proposition}

\begin{proof}
 To see $\mu^\sigma$ as a moment map, we identity $\Lie(G)$ with its dual using the standard hermitian product on $\su_N(\C)$.
Note that $\mu^\sigma$ takes value in $\Lie(G)$ because $\phi(\bs)$ is $T$-invariant, and $\sigma$ commutes with $T$.
We first show that $\mu^\sigma$ is equivariant with respect to the $G$-action. Let $\bs=(s_j)\in\cB^T(V)$ and $g\in G$. As $\om_{\phi(\bs \cdot g)}=\om_{\phi(\bs)}$,
and as the actions of $G$ and $\Aut^T$ commute
 \begin{eqnarray}
	\nonumber
		\mu^\sigma(\bs\cdot g)
		&=&
		i\Hilb(\phi(\bs\cdot g))_0(\sigma\cdot(s_j\cdot g),\sigma\cdot(s_k\cdot g))
	\\
	\nonumber
	&=&
		i\Hilb(\phi_\bs)_0((\sigma\cdot s_j)\cdot g,(\sigma\cdot s_k)\cdot g)
		\\
	\nonumber
	&=&
		Ad_{g^{-1}} (\mu^\sigma(\bs)).
\end{eqnarray}
Thus $\mu^\sigma$ is $Ad$-equivariant. Then, any element $a\in \Lie(G)$ defines
a vector field $\hat{a}$ on $\cB^T(V)$ via $\Phi_\bs$, and
\begin{eqnarray}
\label{eq:musigma a}
	\nonumber
		\langle \mu^\sigma(\bs),a \rangle
		&=&
		 \langle i \Hilb( \sigma^*\phi_\bs)_0,a \rangle
	\\
	&=&
		-\int_X \frac{\sum_{kj}
		( \sigma^*s_k,\sigma^*s_j)_{h_0} \overline{ia_{jk}}}{\sum_k \vert s_k \vert_{h_0}^2} d\mu_{\phi_\bs}.
\end{eqnarray}
On the other hand, from (\ref{eq:theta}),
\begin{equation*}
 \label{eqn:theta}
\Theta_\cB^\sigma(\bs)(\hat{a})=-\int_X \mu_\bs^{a} (1+\Delta_{\phi_\bs})e^{\psi_{\sigma,\phi_\bs}} d\mu_{\phi_\bs}.
\end{equation*}
Then, from the proof of Lemma 3.3.1 in \cite{st}, applying \cite[Equation (15)]{st} to
the path of metrics $\phi_t=\phi(\bs\cdot e^{tiA})$ in $\cH^T$, we have:
\begin{equation*}
 \label{eqn:end}
 \int_X \mu_\bs^{a} (1+\Delta_{\phi_\bs})e^{\psi_{\sigma,\phi_\bs}} d\mu_{\phi_\bs}=
 \int_X \frac{\sum_{jk}
		( \sigma^*s_k,\sigma^*s_j)_{h_0} \overline{ia_{jk}}}{\sum_k \vert s_k \vert_{h_0}^2} d\mu_{\phi_\bs}.
\end{equation*}
Thus
 \begin{equation}
  \label{eq:theta=mu}
  \forall a\in \Lie(G),\; \langle \mu^\sigma(\bs),a \rangle=\Theta_\cB^\sigma(\bs)(\hat{a}).
 \end{equation}
By Proposition \ref{prop:symp}, $\Theta_\cB^\sigma(\hat{a})$ is the momentum for $\hat{a}$ on $Z^T(V)$.
As $\Omega_\cB^\sigma=d\Theta_\cB^\sigma$, Equation (\ref{eq:theta=mu}) tells us that $\mu^\sigma$ is a moment map for the $G$-action on $Z^T(V)$.
\end{proof}

\section{Optimal choice of the weight $\sigma$}
\label{sec:optimalweight}
Through this section, we abbreviate the subscript $k$ if it does not lead to confusion.
With a moment map interpretation for $\sigma$-balanced metrics at hand, assuming the existence
of an extremal metric, we would like to show that for $k$ large enough, there is a $\sigma$-balanced metric on $X$, i.e., a zero of the moment map $\mu^{\sigma}$.
To find such a point, we consider the gradient flow of $I^{ \sigma} \circ \phi$.
General theory of moment maps reduces the problem to an estimate on the first derivative of $\mu^{\sigma}$.
If $T$ is not trivial, we cannot hope to obtain the desired estimate for general $\sigma \in T^{c}$,
and we need to choose $\sigma$ carefully to avoid this obstruction.
Our argument is inspired by \cite{tz} with the viewpoint that $\sigma$-balanced metrics are self-similar for the discrete dynamical system
$$
	\omega \mapsto \omega_{\FS\circ \Hilb(\omega)}.
$$

\subsection{Optimal weight $\sigma$}
First, we introduce an invariant from the derivative of $\mu^{ \sigma}$, which determines the optimal weight $\sigma$.
Fix any $\bs=\{s_{j}\} \in\cB^{T}(V)$ and consider the corresponding K\"ahler metric $\omega_{\phi_{\bs}}$.
For $a= (a_{ij}) \in \Lie(T^{c})$ so that $\{\exp(ita)\}  \subset SL (V)$, define the function $\theta_{a, \bs}$ by
$$
	\theta_{a, \bs}:= \frac{\sum_{jk} (s_{j}, s_{k})_{h_{0}}\overline{i a_{jk}} }{\sum_{j} |s_{j}|_{h_{0}}^{2}}.
$$
For $\sigma\in T^c$, similarly to the modified Futaki invariant, we define
a character:
\begin{equation*}
 \begin{array}{cccc}
\cF^{ \sigma}: & \Lie(T^{c}) &  \to \C \\
	 & a & \mapsto & - \int_{X} \theta_{a, \bs} (1 + \Delta_{ \omega_{\phi_{\bs}}})e^{\psi_{\sigma, \phi_{\bs}}} d\mu_{\phi_{\bs}}.
 \end{array}
\end{equation*}
\begin{proposition}
\label{prop:futakimorita}
The map  $\cF^{ \sigma}$ satisfies:
\begin{enumerate}
\item\label{prop:fm2} If $\bs$ is $\sigma$-balanced, then $\cF^{ \sigma}$ vanishes. More precisely:
$$
	\frac{d}{dt}\bigg\vert_{t=0} I^{ \sigma} \circ \phi (\bs \cdot e^{ita})=\langle \mu^{ \sigma } (\bs ), a \rangle = \cF^{ \sigma}(a).
$$
\item\label{prop:fm1} $\cF^{ \sigma}$ is independent of the choice of basis $\bs$.
\item\label{prop:fm3} There exists a unique $\sigma \in T^{c}$, such that $\cF^{ \sigma}$ vanishes.
\end{enumerate}
\end{proposition}

\begin{proof}
The statement (\ref{prop:fm2}) follows directly by definition and the computations of Section \ref{subsec:momentmap}.
For an element $a \in \Lie(T^{c})$, define the function $\widetilde{\theta}_{a, \bs}$
\begin{equation}\label{eq:normalizedpotential}
	\iota_{\hat{a}} \omega_{\phi_{\bs}}	= i \bar{\partial} \widetilde{\theta}_{a, \bs},
	\,\,
	\int_{X} \widetilde{\theta}_{a, \bs}d\mu_{\phi_{\bs}}=0.
\end{equation}
It is known that $\theta_{a, \phi_{\bs}} - \widetilde{\theta}_{a, \phi_{\bs}}$ is equal to a constant independent of the choice of $\bs$ (cf. the proof of Lemma 3.4 \cite{futaki04}).
Then, it is sufficient to show that
$$
	\widetilde{\cF}^{ \sigma} (a)= - \int_{X} \widetilde{\theta}_{a, \bs} (1 + \Delta_{ \omega_{\phi_{\bs}}})e^{\psi_{\sigma, \phi_{\bs}}} d\mu_{\phi_{\bs}}
$$
is independent of the choice of $\bs$.
Take another basis $\bs'$ and connect it with $\bs$ by a geodesic $\{\exp(i t \xi) \mid\, t \in [0,1]\}$.
In what follows, we omit subscriptions.
By direct calculation, using the equality $\dot{\psi} = ( \nabla \psi, \nabla \dot{\phi} )$ (cf. Lemma 4.2 \cite{st}):
\begin{eqnarray}
  \nonumber
		- \frac{d}{dt} \widetilde{\cF}^{ \sigma} (a)
	&=&
		\int_{X} \big((1+ \Delta) \dot{\widetilde{\theta}}\big) e^{\psi } d\mu
		+
		\int_{X} (\nabla \bar{\nabla} \dot{\phi}, \nabla \bar{\nabla} \widetilde{\theta}) e^{\psi } d\mu
	\\
  \nonumber
	&& \quad
		+
		\int_{X} \big((1+ \Delta) \dot{\widetilde{\theta}}\big)(\nabla \psi, \nabla \dot{\phi}) e^{\psi } d\mu
		-
		\int_{X} \big((1+ \Delta) \dot{\widetilde{\theta}}\big) e^{\psi } \Delta  \dot{\phi} d\mu
	\\
  \nonumber
	&=&
		\int_{X} \big((1+ \Delta) \dot{\widetilde{\theta}}\big) e^{\psi } d\mu
		-
		\int_{X} (\nabla  \dot{\phi}, \nabla \widetilde{\theta}) e^{\psi } d\mu
		-
		\int_{X} (\nabla \bar{\nabla} \widetilde{\theta}, \nabla \dot{\phi} \bar{\nabla} \psi) e^{\psi } d\mu
	\\
  \label{eq:derivativeoftildeF}
	&=&
		\int_{X} (1+ \Delta) \big( \dot{\widetilde{\theta}} - (\nabla  \dot{\phi}, \nabla \widetilde{\theta}) \big) e^{\psi } d\mu.
\end{eqnarray}
The equality  (\ref{eq:normalizedpotential}) implies that $\dot{\widetilde{\theta}} - (\nabla  \dot{\phi}, \nabla \widetilde{\theta}) $ is constant on $X$.
The second equality in (\ref{eq:normalizedpotential}) that this constant is zero, in fact
$$
	\int_{X} \big( \dot{\widetilde{\theta}} - (\nabla  \dot{\phi}, \nabla \widetilde{\theta}) \big) d\mu
	=
	\frac{d}{dt} \int_{X} \widetilde{\theta} d\mu =0.
$$
Hence, $\frac{d}{dt}\widetilde{\cF}^{ \sigma} (a)$ vanishes identically.
The proof of (\ref{prop:fm1}) is completed.
We will show (\ref{prop:fm3}).
Assume that $a$ is diagonal by a change of basis $\bs$.
We denote $\tau(t)= \exp(it a)$.
By definition,
\begin{eqnarray*}
		\cF^{ \sigma} (a)
	&=&
		- \int_{X} \big((1+ \Delta_{ \omega_{\phi_{\bs}}}) \theta_{a, \bs}\big) e^{\psi_{\sigma, \phi_{\bs}}} d\mu_{\phi_{\bs}}
	\\
	&=&
		\frac{d}{dt}\bigg|_{t=0} \int_{X} \frac{\sum_{i} |\sigma^{*}s_{i}  |^{2}}{\sum_{j}|s_{j} \cdot \tau (t)|^{2}} d\mu_{\phi_{\bs \cdot \tau (t) }}
	\\
	&=&
		\frac{d}{dt}\bigg|_{t=0} \int_{X} \frac{\sum_{i} |s_{i}\cdot \sigma \cdot \tau(-t) |^{2}}{\sum_{j}|s_{j}|^{2}} d\mu_{\phi_{\bs  }},
\end{eqnarray*}
because $\tau (t)$ commutes with $\sigma$.
For a given $\bs$, we define the functional $\cG: T^{c} \to \R$ by
$$
	\cG (\sigma):=  \int_{X} \frac{\sum_{i} |s_{i} \cdot \sigma  |^{2}}{\sum_{j} |s_{j}|^{2}} d \mu_{\bs}.
$$
By direct calculation, we find that $\cG$ is independent of $\bs$.
Obviously, $\cG$ is strongly convex and proper on $T^{c}$.
Hence, $\cG$ has a unique critical point $\sigma_{0} \in T^{c}$ independent of the choice of $\bs$.
Considering the first variation of $\cG$, we find that $\cF^{ \sigma_{0}}$ vanishes on $\Lie(T^{c})$.
The proof is completed.
\end{proof}

\begin{definition}
\label{def:optimal weight}
The weight $\sigma_k\in T^c$ such that $\cF^{\sigma_k}=0$ will be called the {\it $k^{th}$ optimal weight} (or optimal weight).
\end{definition}

\begin{remark}
If $\sigma= \mathrm{id}$, then the vanishing of $\cF^{ \mathrm{id}}$ is equivalent to the condition in Theorem A \cite{ma04-1} called stability of isotropy actions for $(X,L)$.
It is also equivalent to the vanishing of higher Futaki invariants \cite{futaki04}.
\end{remark}

\subsection{Convergence of the optimal weights $\sigma_k$}
Next, we will show that the optimal weights $\sigma_{k}$ approximate the extremal vector field $v_{ex}$.
\begin{proposition}
\label{prop:weight}
Let $\sigma_{k} \in T^{c}_{k}$ be the optimal weight as above.
Take $v_{k} \in \Lie (T^{c})$ so that $\exp(v_{k})= \sigma_{k}$.
Then, we have
 $$
 \lim_{k \rightarrow \infty} k v_{k} = v_{ex}
 $$
 where $v_{ex}$ is the extremal vector field of $(X,L)$ with respect to the torus $T$.
\end{proposition}
First we construct an approximate weight for each $k$.
\begin{lemma}
\label{lem:appro_vf}
There exist a constant $c\in \R$ and vector fields $\nu_{j} \in \Lie (T^{c})$$(j \ge 2)$ such that  for any $q \ge 2$ there exists a constant $C_{q}$ with the following property:
let
\begin{equation}\label{eq:app_vf1}
	v_{q}(k):= c v_{ex} k^{-1} + \sum_{j=2}^{q} \nu_{j}k^{-j},
\end{equation}
then we have
\begin{equation}\label{eq:app_vf2}
	|\cF^{ \widetilde{\sigma}_{k}}(u)| < C_{q} k^{-q-1}\max_{i} |b_{i}|
	\quad
	\mbox{for } k \gg 0
\end{equation}
where $\widetilde{\sigma}_{k}:= \exp (v_{q}(k)) \in T^{c}_{k}$ and $u \in \Lie (T^{c})$ is written by
$$
	u= \mathrm{diag}  \frac{1}{2}(b_{1}, \ldots, b_{N_{k}}) \in \Lie (T^{c}_{k})
$$
with respect to $\Hilb_{k}(\omega)$.
\end{lemma}
\begin{proof}
We denote by $\omega_{k}$ the K\"ahler form $\omega_{\FS_{k}\circ \Hilb_{k}(\omega)}$ through the proof.
For a given $q$ and a vector field as in (\ref{eq:app_vf1})
$$
	v_{q}(k)= \sum_{j=1}^{q} \nu_{j}k^{-j}
$$
let
$$
	\widetilde{\sigma}_{k}(t):= \exp (tv_{q}(k))
$$
and $\widetilde{\sigma}_{k}:= \exp (1 \cdot v_{q}(k))$.
Take any $u\in \Lie (T^{c})$ and let $\tau_k (t):= \exp (t u) \in T^{c}_{k}$.
From (\ref{eq:normalization_phi}), (\ref{eq:normalizedpotential}) and that $\theta_{u,\omega_{k}}-\widetilde{\theta}_{u,\omega_{k}}$ is constant, we find that
\begin{eqnarray}
	\nonumber
		\mathcal{F}^{\widetilde{\sigma}_{k}}(u)-\widetilde{\mathcal{F}}^{\widetilde{\sigma}_{k}}(u)
	&=&
		-\int_{X} (\theta_{u,\omega_{k}}-\widetilde{\theta}_{u,\omega_{k}})e^{\psi_{ \widetilde{\sigma}_{k}, \omega_{k}}} d\mu_{ \omega_{k}}
	\\
	\nonumber
	&=&
		-\frac{N_k}{k^n}\frac{1}{\vol(X)}\int_{X} (\theta_{u,\omega_{k}}-\widetilde{\theta}_{u,\omega_{k}})d\mu_{ \omega_{k}}
	\\
	\label{eq:F_sigma_k_zero}
	&=&
	-\frac{N_k}{k^n}\frac{1}{\vol(X)}\int_X \theta_{u,\omega_{k}}d\mu_{ \omega_{k}}
\end{eqnarray}
for each $k$. We now give some expansions for $\widetilde{\mathcal{F}}^{\widetilde{\sigma}_{k}}(u)$ and $\frac{1}{\vol(X)}\int_X \theta_{u,\omega_{k}}d\mu_{ \omega_{k}}$.
Note that for any $1$-parameter subgroup $\lbrace\sigma(t)=\exp(t\nu)\rbrace\subset T_k^c $, we have an expansion
\begin{equation}
 \label{eq:expansion sigma(t)omega}
\sigma(t)^*\om_k=\om_k+ i\del\delb\big(\sum_{j=1}^{q}\widetilde{\theta}^j_{\nu,\omega_{k}}t^{j}\big)+\mathcal{O}(t^{q+1})
\end{equation}
with
$$
	\int_{X} \widetilde{\theta}^j_{\nu, \omega_{k}}d\mu_{\omega_{k}}=0\:,\: j=1\ldots q.
$$
As the vector fields $\nu_j$ commute, so do the automorphisms $\exp(\nu_{j}k^{-j})$.
Thus, by definition of $v_{q}(k)$, and using (\ref{eq:expansion sigma(t)omega}),
\begin{eqnarray}
	\nonumber
		i\partial\bar\partial\psi_{\widetilde{\sigma}_{k}, \omega_{k}}
	&=&
		\bigg(\frac{1}{k}\bigg)\frac{(\widetilde{\sigma}_{k}(1/k))^*\omega_{k}-\omega_{k}}{k^{-1}}
	\\
	\label{eq:expansionpsi}
	&=&
		i\partial\bar\partial
		\big(
			 \sum_{j=1}^{q}(\widetilde{\theta}_{\nu_{j},\omega_{k}}+\theta'_j)k^{-j}
		\big)
		+\mathcal{O}(k^{-q-1}),
\end{eqnarray}
where
$$
	\iota_{\nu_{j}} \omega_{k}	= i \bar{\partial} \widetilde{\theta}_{\nu_{j}, \omega_{k}},
	\,\,
	\int_{X} \widetilde{\theta}_{\nu_{j}, \omega_{k}}d\mu_{\omega_{k}}=0
$$
and the functions $\theta'_j$, $j=1\ldots q$, only depend on $\om_k$ and the vector fields $(\nu_1,\cdots,\nu_{j-1} )$.
Then, we have
\begin{equation}
	\label{eq:tilde_F_product}
		\widetilde{\mathcal{F}}^{\widetilde{\sigma}_{k}}(u)
	=
		-\int_X
		\widetilde{\theta}_{u, \omega_{k}}
		(1+ \Delta_{ \omega_{k}})
		\big(1+
		\sum_{j=1}^{q}(\widetilde{\theta}_{\nu_{j},\omega_{k}}+P_j(\widetilde{\theta}_{\nu_{i},\omega_{k}},\theta'_i))k^{-j}
		\big)
		d\mu_{ \omega_{k}}
		+\mathcal{O}(k^{-q-1})
\end{equation}
with $P_j(\widetilde{\theta}_{\nu_{i},\omega_{k}},\theta'_i)$ polynomial in $(\widetilde{\theta}_{\nu_{1},\omega_{k}},\cdots,\widetilde{\theta}_{\nu_{j-1},\omega_{k}},\theta'_1,\cdots,\theta'_{j-1})$.
  Note that the coefficients of $k^{-j}\, (j\ge 1)$ in (\ref{eq:tilde_F_product}) are independent of the choice of $\omega$.
  This follows from the calculation (\ref{eq:derivativeoftildeF}).
  In fact, denoting the expansion
  $$
    \exp(\psi_{ \sigma_k, \omega_k})
    =
    1+
    \sum_{j \ge 1} \widetilde{\Theta}_j k^{-j},
  $$
  then the calculation (\ref{eq:derivativeoftildeF}) tells us
  \begin{eqnarray*}
      \frac{d}{ds} \int_X
      \widetilde{\theta}_{u, \omega_{k}}(1+ \Delta_{ \omega_{k}})
      \widetilde{\Theta}_j d\mu_{ \omega_k}
    &=&
      \int_X (1+ \Delta_k)\big( \frac{\partial \widetilde{\theta}_{u, \omega_{k}}}{ \partial s} - (\nabla \frac{\partial \phi}{\partial s}, \nabla \widetilde{\theta}_{u, \omega_{k}})\big)
      \widetilde{\Theta}_j d\mu_{ \omega_k}
    \\
    &=&
      0
  \end{eqnarray*}
  where $s$ parametrizes a perturbation of the K\"ahler form.
On the other hand,
according to Proposition 2.2.2 in \cite{Don02},
we find that
\begin{equation}\label{eq:theta_futaki}
	\frac{1}{\vol(X)}\int_X \theta_{u,\omega_{k}}d\mu_{ \omega_{k}}
	= F_{1}(u) k^{-1} + F_{2}(u)k^{-2} + \cdots + F_{q}(u)k^{-q} + \cO(k^{-q-1})
\end{equation}
where $F_{j}$ is a Lie algebra homomorphism on $\Lie(T^{c})$ to $\R$.
In particular, $F_{1}$ is equal, up to a multiplicative constant, to the original Futaki invariant \cite{fut}
$$
	\mathrm{Fut}(u)=\int_X \widetilde{\theta}_{u,\omega}
	S(\omega)d\mu_{ \omega}.
$$
From (\ref{eq:F_sigma_k_zero}), (\ref{eq:tilde_F_product}) and (\ref{eq:theta_futaki}), we have
\begin{eqnarray}
  \nonumber
		\mathcal{F}^{\widetilde{\sigma}_{k}}(u)
	&=&
		-\bigg(\frac{N_k}{k^n}F_{1}(u)+ \int_{X} \widetilde{\theta}_{u, \omega_{k}}\widetilde{\theta}_{\nu_{1}, \omega_{k}} d\mu_{ \omega_{k}}\bigg)k^{-1}
	\\
  \nonumber
	&&
		-\sum_{j=2}^{q}\bigg(\frac{N_k}{k^n}F_{j}(u)+ \int_{X} \widetilde{\theta}_{u, \omega_{k}}(\widetilde{\theta}_{\nu_{j}, \omega_{k}}+P_j(\widetilde{\theta}_{\nu_{i},\omega_{k}},\theta'_i))  d\mu_{ \omega_{k}}\bigg)k^{-j}
	\\
  \nonumber
	&&
		-\sum_{j=2}^{q}\bigg(\int_{X} \widetilde{\theta}_{u, \omega_{k}} (\Delta_{ \omega_{k}}(\widetilde{\theta}_{\nu_{j-1}, \omega_{k}}+P_{j-1}(\widetilde{\theta}_{\nu_{i},\omega_{k}},\theta'_i))) d\mu_{ \omega_{k}}\bigg)k^{-j}
	\\
  \label{eq:expansion_F_tilde_sigma}
	&&
		\quad
		- \sum_{j=q+1}^{\infty}\frac{N_k}{k^n}F_{j}(u) k^{-j} +\mathcal{O}(k^{-q-1}).
\end{eqnarray}
Recall that for holomorphy potentials $\widetilde{\theta}_{1}, \widetilde{\theta}_{2}$ under the normalization
$$
	\int_{X} \widetilde{\theta}_{i} d \mu=0,
$$
the bilinear form (\ref{eq:bilinearfutakimabuchi}) is non-degenerate (see \cite{fm}).
Then, we can construct $\nu_{j}$ inductively in $j$ so that for $j=1$
$$
	\frac{N_k}{k^n}F_{1}+ \int_{X} \widetilde{\theta}_{u, \omega_{k}}\widetilde{\theta}_{\nu_{1}, \omega_{k}} d\mu_{ \omega_{k}}=0
$$
and for $2 \le j \le q-1$
\begin{eqnarray}
\nonumber
    \int_{X} \widetilde{\theta}_{u, \omega_{k}}\widetilde{\theta}_{\nu_{j}, \omega_{k}} d\mu_{ \omega_{k}}
     &=&
    - \frac{N_k}{k^n}F_{j}
    \\
    \nonumber
    & &
    -    \int_{X} \widetilde{\theta}_{u, \omega_{k}}  P_j(\widetilde{\theta}_{\nu_{i},\omega_{k}},\theta'_i)
        d\mu_{ \omega_{k}}
     \\
    \nonumber
    & &
     -    \int_{X} \widetilde{\theta}_{u, \omega_{k}}  \Delta_{ \omega_{k}}(\widetilde{\theta}_{\nu_{j-1}, \omega_{k}}+P_{j-1}(\widetilde{\theta}_{\nu_{i},\omega_{k}},\theta'_i))
    d\mu_{ \omega_{k}}.
\end{eqnarray}
In particular, by definition of extremal vector fields, $\nu_{1}$ is equal to the extremal vector field.
We remark that the right hand side in the above equality are independent of the choice of $\omega$ due to \cite{fm2}
and the independency of the $k^{-j}\, (j\ge 1)$ coefficients in (\ref{eq:tilde_F_product}) .
Finally, we prove the inequality (\ref{eq:app_vf2}).
The construction of $\nu_j$ implies
\begin{equation}
 \label{eq:rest of Ftilde}
	|\mathcal{F}^{\widetilde{\sigma}_{k}}(u)|
	=
	\bigg|\sum_{j=q+1}^{ \infty} \bigg(\frac{N_k}{k^n}F_{j}(u)
  + \int_X \widetilde{\theta}_{u, \omega_{k}} (\widetilde{\Theta}_j + \Delta_k \widetilde{\Theta}_{j-1})d\mu_{ \omega_k}\bigg)k^{-j}\bigg| .
\end{equation}
As we have seen, the coefficients of $k^{-j}$ in the right hand side of (\ref{eq:rest of Ftilde})
are independent of $\omega$, i.e., Lie algebra homomorphisms on the finite dimensional $\Lie (T^c_k)$.
Then there exists a constant $C$ such that
$$
  \bigg|
  \frac{N_k}{k^n}F_{j}(u)
  + \int_X \widetilde{\theta}_{u, \omega_{k}} (\widetilde{\Theta}_j + \Delta_k \widetilde{\Theta}_{j-1})d\mu_{ \omega_k}
  \bigg|
  \le C\max_i|b_i|
$$
for large enough $k$.
Hence we prove (\ref{eq:app_vf2}).
The proof is completed.
\end{proof}
\begin{proof}[Proof of Proposition \ref{prop:weight}]
We perturb an approximate weight from Lemma \ref{lem:appro_vf} to the one we desired.
Let $v(k)$ be the vector field in Lemma \ref{lem:appro_vf} for  $q= 3n+3 \ge 3\log_{k}N_k+2$.
Then, we connect $\sigma_k$ to $\widetilde{\sigma}_{k}$ by a path
$$
	\zeta_k (t)=\exp(\mathrm{diag} ( \frac{1}{2}b_{i}t))
$$
for some basis $\bs$ for $\Hilb_{k}(\omega)$.
As seen in the proof of Proposition \ref{prop:futakimorita} (\ref{prop:fm2}),
\begin{eqnarray}
	\nonumber
		\mathcal{F}^{\sigma_{k}}(v (k))
	&=&
		\int^1_0 dt \int_X \frac{\sum _{\alpha} |b _{\alpha}|^2 e ^{- tb _{\alpha}}|s _{\alpha}|^2}
		{\sum _{\beta} |s _{\beta}|^2} d\mu_{ \omega_{k}}
	\\
	\label{eq:est1}
	&\ge&
		|b _{\alpha_0}|^2 \int_X \frac{|s _{\alpha_0}|^2}{\sum _{\beta} |s _{\beta}|^2} d\mu_{ \omega_{k}}
	\\
	\label{eq:est2}
	&\ge&
		 \frac{C}{N_k}|b _{\alpha_0}|^2
	\\
	\label{eq:est3}
	&\ge&
		\frac{C}{(N_k)^3} \max |b _{\alpha}|^2
\end{eqnarray}
for sufficiently large $k$.
In (\ref{eq:est1}), the subscript $\alpha_0$ is determined so that $b _{\alpha_0}= \inf b _{\alpha} <0$.
The inequality (\ref{eq:est2}) follows from
\begin{eqnarray*}
		\int_X \frac{|s _{\alpha_0}|^2}{\sum _{\beta} |s _{\beta}|^2}
		d\mu_{ \omega_{k}}
	&\ge&
		C \int_X  \frac{|s _{\alpha_0}|^2_{h^k}}{\sum _{\beta} |s _{\beta}|^2_{h^k}}
		d\mu_{ \omega}
	\\
	&\ge&
		C \frac{1}{\max_X \rho_k (\omega)}
	\\
	&\ge&
		C \frac{1}{ \max_X (k^n + A_1 (\omega) k ^{n-1}+ \cdots)}
	\\
	&\ge&
		\frac{C}{N_k}.
\end{eqnarray*}
Recall that $\rho_k (\omega)$ denotes by the $k$-th Bergman function of $\omega$.
In the first line, we use $\omega (k) \to \omega$ as $k \to \infty$.
In the third line, we use the fact that $\max_X |A_i (\omega)|$ is independent of $\omega$.
The inequality (\ref{eq:est3}) follows from
$$
	|b _{\alpha_0}| \ge \frac{1}{N_k-1}  \max |b _{\alpha}|,
$$
because $\sum _{\alpha} b _{\alpha} =0$.
From (\ref{eq:app_vf2}) and (\ref{eq:est3}), we have
\begin{equation}\label{eq:estimateb}
	\max _{\alpha} |b _{\alpha}| < Ck ^{-2}.
\end{equation}
This implies that
$$
	k (v(k) - v_k) \to 0
$$
as $k \to \infty$.
The proof of Proposition \ref{prop:weight} is completed.
\end{proof}

\section{Proof of Theorem \ref{thm:A}}
\label{sec:proof}

Start with a polarized $(X,L)$ with an extremal metric $\om_{ex}\in 2 \pi c_1(L)$.
Here $\sigma_k$ denotes the optimal weight as defined in Definition \ref{def:optimal weight}.
The key proposition that we want to use is the following (see \cite{Don01} and \cite{sz10}):

\begin{proposition}
\label{prop:Dmmap}
Let $(Z,\om_Z)$ be a (finite dimensional) K\"ahler manifold with a Hamiltonian $G$-action,
for a compact Lie group $G$.
Denote by $\mu$ the moment map for this action. For each $x\in Z$, let $G_x$ be the stabilizer of $x$ in $G$,
and let $\Lambda_x^{-1}$ denote the operator norm of
$$
\sigma_x^*\circ \sigma_x : \Lie(G_x)^\perp \rightarrow \Lie(G_x)^\perp
$$
where $\sigma_x : \Lie(G_x)^\perp \rightarrow T_xZ$ is the infinitesimal action of $G$ at $x$ and the orthogonal
complement is computed with respect to an invariant scalar product on $\Lie(G)$.
Let $x_0\in Z $ with $ \mu(x_{0}) \in \Lie(G_{x_0})^\perp$.
Assume that there are real numbers $\lambda,\delta$ such that
\begin{itemize}
\item $\Lambda_x\leq \lambda$
for all $x=e^{i\xi}x_0$ with $\vert\vert \xi \vert \vert < \delta$, and
\item $\lambda \vert \vert \mu(x_0) \vert \vert < \delta$.
\end{itemize}
Then there exists $y=e^{i\eta}x_0$ such that $\mu(y)=0$, with $\vert \vert \eta\vert \vert \leq \lambda \vert\vert \mu(x_0)\vert\vert$.
\end{proposition}
\noindent We want to use Proposition \ref{prop:Dmmap} with the moment map setting of Section \ref{sec:mmap} to build $\sigma_k$-balanced
metrics. This will rely on two steps. In Section \ref{sec:control of derivative} we will estimate the norm of $\Lambda_x$ in our setting.
Then in Section \ref{sec:construction} we will construct an ``almost $\sigma_k$-balanced metric'' for $k$ large
enough.

\subsection{Control of the derivative of $\mu^\sigma$}
\label{sec:control of derivative}
We specialize to our setting, considering $(Z,\om_Z)=(Z^T(V_k),\Om_Z^{\sigma_k})$ and $G=G_k$.
To state the main result of this section, we introduce some definitions and notations.

\noindent A basis $\bs\in\cB^T(V_k)$ is said to have $R$-bounded geometry, $R>0$, if the corresponding metric
$\om_{\phi_\bs}$ satisfies
 \begin{itemize}
  \item[i)] $\om_{\phi_\bs} > R^{-1}k\om_0$
  \item[ii)] $\vert\vert \om_{\phi_\bs} - k\om_0\vert\vert_{C^r(k\om_0)} < R$
 \end{itemize}
for $C^r(k\om_0)$ the $C^r$-norm defined by $k\om_0$.
For a basis $\bs\in\cB^T(V_k)$, we set $\Lambda_\bs$ as in Proposition \ref{prop:Dmmap}. 
Last, we denote the operator norm of a matrix $A$ by $\|A\|_{op}$:
$$
	\|A\|_{ \mathrm{op}}:= \max \frac{|A(\xi)|}{|\xi|}.
$$
The aim of this section is to prove:
\begin{proposition}
\label{prop:controldmu}
 For any $R>1$, there are positive constants $C$ and $\ep<\frac{1}{10}$ such that, for any $k$, if the basis $\bs\in\cB^T(V_k)$
 has $R$-bounded geometry, and if $\vert\vert\mu^{\sigma_k}(\bs)\vert\vert_{ \mathrm{op}}<\ep$, then
\begin{equation}\label{eq:upperboundLambda}
 \Lambda_\bs \leq C  k^2.
\end{equation}
\end{proposition}
\noindent
This is a generalization of Theorem 21 in \cite{Don01} and Theorem 2 in \cite{ps04}.
The estimate  (\ref{eq:upperboundLambda}) in Proposition \ref{prop:controldmu} is equivalent to
$$
\forall \xi\in \Lie(G_\bs)^\perp\;,\; g_{Z_k}^{\sigma_k}(\xi,\xi) \geq C k^{-2} \vert\vert \xi \vert \vert^2
$$
where $G_\bs$ is the stabilizer of $\bs$ in $G_k$.
Note that for each $\bs\in Z^T(V_k)$,
$$
	G_\bs^c=\Aut^T_k
$$
and
$$
	\Aut_k^T\simeq  T^c
$$
as $T$ is a maximal torus. Thus for any $\bs\in Z^T(V_k)$,
$$
	\Lie(G_\bs)^\perp\simeq \Lie(T)^\perp,
$$
and (\ref{eq:upperboundLambda}) is equivalent to
\begin{equation}
 \label{eq:upperbound reformulation}
\forall \xi\in \Lie(T)^\perp\;,\; g_{Z_k}^{\sigma_k}(\xi,\xi) \geq C k^{-2} \vert\vert \xi \vert \vert^2.
\end{equation}
We need a more explicit formulation for $g_{Z_k}^{\sigma_k}$.
For this, we introduce a metric on the tangent bundle $T \mathbb{CP}^{N_k-1}$ of $\mathbb{CP}^{N_k-1}$ as follows.
We recall that for any $\xi\in \Lie(G_k)$, $\hat{\xi}$ is the induced vector field on $\C\P^{N_k-1}$.
Let $\mathbf{z}=[Z_1, \cdots , Z_{N_k}]$ be the homogeneous coordinates on $\mathbb{CP}^{N_k-1}$.
Up to a change of coordinates preserving the weight decomposition under the $T^c$-action,
we can assume that $\hat{\xi}= \frac{d}{dt}\mid_{t=0} \exp (t \tau)$ where $\tau$ is diagonlized by
$$
  \tau = \mathrm{diag} (b_1, \cdots, b_{N_k}).
$$
Then, we define a metric $|\cdot|_{\sigma_k\cdot h_{FS}}$ on $T \mathbb{CP}^{N_k-1}$ by
$$
  |\hat{\xi}|^2_{\sigma_k\cdot h_{FS}}(\mathbf{z}):=
  \sum_{\alpha=1}^{N_k}|(b_\alpha- \hat{\phi}'_0)(\sigma_k\cdot Z_\alpha)|^2_{FS}
  =
  \sum_{\alpha=1}^{N_k}\frac{|(b_\alpha- \hat{\phi}'_0)(\sigma_k\cdot Z_\alpha)|^2}{\sum_{\alpha=1}^{N_k}|Z_{\alpha}|^2}
$$
where
$$
  \hat{\phi}(t)= \log (\mathbf{z}\exp(t\tau) \mathbf{z}^*)= \log \bigg( \sum_{\alpha} |\exp(\frac{1}{2}t\tau)Z_\alpha|^2 \bigg).
$$
Remark that if $\sigma_k$ is trivial, then $|\cdot|_{\sigma_k\cdot h_{FS}}$ is the ordinary Fubini-Study metric.
With respect to the metric $|\cdot|_{\sigma_k\cdot h_{FS}}$, we will denote $\pi_N\hat{\xi}$ the orthogonal projection onto the orthogonal of $TX$ in $\Phi_\bs^*TX$, where $\Phi_\bs$ is defined by (\ref{eq:isomPhis}).
Now we have the following description of $g_{\cB_k}^{\sigma_k}(\bs)$ corresponding to \cite[Equation (5.6)]{ps04} in our setting.
\begin{lemma}
 \label{lem:gZ}
  Let $\xi\in \Lie(G_k)$.
 Then
 $$
 g_{\cB_k}^{\sigma_k}(\bs)(\xi,\xi)=\frac{1}{V}\int_X |\pi_N\hat{\xi}|^2_{\sigma_k\cdot h_{FS}} d\mu_{\phi_\bs}.
 $$
\end{lemma}

\begin{proof}
This follows directly from proof of Lemma 3.8 in \cite{st}.
We abbreviate the subscript $k$ for $\sigma_k$.
Let $\bs(t)=\bs \cdot\exp(t \xi)$ and $\bs(t)^c=\bs \cdot\exp(tJ_{\cB_k} \xi)$. As $\xi\in\Lie(G_k)$, $e^{t\xi}\in G_k$
and $\om_\phi(\bs(t))=\om_\phi(\bs(0))$. Then from (\ref{eq:gB}),
$$
g_{\cB_k}^{\sigma_k}(\xi,\xi)=\dfrac{d^2}{dt^2}\bigg\vert_{t=0} (I^{\sigma}\circ\phi )(\bs(t)^c).
$$
From (3.8) and (3.9) in \cite{st}, the right hand side in the above equality is equal to
$$
  \frac{1}{V} \int_X \sum_{\alpha=1}^{N_k}
  |(b_\alpha-\phi'_0)(\sigma_k\cdot s_\alpha)- (\nabla \phi'_0, \nabla(\sigma_k \cdot s_\alpha))|^2_{FS(\bs(0))} d\mu_{\phi_\bs}
$$
where
$$
  \phi_t= \log (\bs(t)^c(\bs(t)^c)^*).
$$
By definition, the integrand in the above is just $|\pi_N\hat{\xi}|^2_{\sigma_k\cdot h_{FS}}$, because
$$
  \iota_{\pi_T \hat{\xi}} \omega_{\phi(\bs(0))}= i \overline{\partial} \phi'_0.
$$
\end{proof}

\noindent To obtain the estimate  (\ref{eq:upperboundLambda}) of Proposition \ref{prop:controldmu}, we use the following uniform control:

\begin{lemma}
\label{lem:boundedweights}
Let $\sigma_k$ be the optimal weight, and denote by $[\sigma_k]$ the matrix representing $\sigma_k$ in any
basis $\bs\in \cB^T(V_k)$.
There exists a constant $c>0$ such that for sufficiently large  $k$,
$$
c^{-1} <\inf_{ij} \vert  [\sigma_k]_{ij} \vert \leq \sup_{ij} \vert  [\sigma_k]_{ij} \vert<c
$$
\end{lemma}

\begin{proof}
Let $\theta_{v_{k}, \bs}$ be the holomorphy potential of the vector field $v_{k}$ satisfying $\sigma_{k}= \exp(v_{k})$ as in (\ref{eq:potentialfunctiontheta}).
Recall that $\theta_{v_{k}, \bs}$ defines the lifted action of $v_{k}$ on $L$.
Then, we find that there exists a constant $c>0$ such that
\begin{equation}\label{eq:supinfsigma}
	\exp(c\min \theta_{v_{k}, \bs})
	<
	\inf_{ij} \vert  [\sigma_k]_{ij} \vert \leq \sup_{ij} \vert  [\sigma_k]_{ij} \vert
	<
	\exp(c\max \theta_{v_{k}, \bs})
\end{equation}
for sufficiently large $k$.
From the theory of moment maps, both of $\max \theta_{v_{k}, \omega_{\bs}}$ and $\min \theta_{v_{k}, \omega_{\bs}}$ are independent of the choice of $\bs$.
In fact, they are determined by the image of the moment map $\mu: X \to \R=\Lie (S^{1})$ with respect to the $S^{1}$-action on $(X, \omega_{\bs})$ induced by $v_{k}$.
On the other hand,  from the normalization (\ref{eq:normalizationthetasigma}) and Proposition \ref{prop:weight}, we find that
$$
\theta_{v_{k}, \bs_{k}} \to \theta_{v_{ex}, \omega}
$$
as $k \to \infty$ for a given $\omega$, where $\theta_{v_{ex}, \omega}$ is the potential function satisfying the normalization (\ref{eq:normalizationthetasigma}) and $\bs_{k}$ is an orthonormal basis with respect to $\Hilb_{k}(\omega)$.
Since the maximum and minimum of $\theta_{v_{ex}, \omega}$ are  also independent of $\omega$, for sufficiently large $k$, there exists $C>0$ such that
\begin{equation}\label{eq:maxminthetaex}
	C^{-1} < \exp(\min \theta_{v_{k}, \bs})
	\le
	\exp(\max \theta_{v_{k}, \bs})
	< C
\end{equation}
for any $\bs$ and $k$.
The equalities (\ref{eq:supinfsigma}) and (\ref{eq:maxminthetaex}) complete the proof.
\end{proof}

\begin{proof}[Proof of Proposition \ref{prop:controldmu}]
We follow \cite[Theorem 2]{ps04} with a modification for our setting.
We want to show :
$$
\forall \xi\in \Lie(T)^\perp\;,\; g_{Z_k}^{\sigma_k}(\xi,\xi) \geq C k^{-2} \vert\vert \xi \vert \vert^2.
$$
The above inequality is derived from:
\begin{eqnarray}
		\label{eq:firsteqPS}
		\quad \vert\vert \xi\vert\vert ^2
	&\leq&
		c_R'k\vert\vert \hat{\xi} \vert\vert ^2
	\\
	\label{eq:secondeqPS}
\vert\vert \hat{\xi} \vert\vert ^2	& =&
		\vert\vert \pi_T\hat{\xi} \vert\vert ^2 + \vert\vert \pi_N\hat{\xi} \vert\vert ^2
	\\
  \label{eq:lasteqPS}
c_R \vert\vert \pi_T\hat{\xi} \vert\vert ^2	&\leq & k\vert\vert \pi_N\hat{\xi} \vert\vert ^2.
\end{eqnarray}
Here, for a vector field $\hat{\xi}$ on $T \mathbb{CP}^{N_k-1}$, $\|\hat{\xi}\|^2$ denotes the $L^2$ norm with respect to the volume form $d\mu_{\phi_\bs}$ on the base and the twisted Fubini-Study metric $|\cdot|_{\sigma_k\cdot h_{FS}}$ on the fiber.
These inequalities can be obtained as \cite[Equation (5.7), (5.8), (5.9)]{ps04}.
We only highlight the differences in the arguments to obtain (\ref{eq:firsteqPS}) and (\ref{eq:lasteqPS}).

\noindent Proof of (\ref{eq:firsteqPS}):
by definition, we have
\begin{equation}\label{eq:decomposition_fibernorm}
  |\hat{\xi}|^2_{\sigma_k\cdot h_{FS}}=
  \frac{(\sigma \cdot \bs)\xi\xi^* (\sigma\cdot \bs)^*}{\bs\bs^*}
  -2 \frac{((\sigma \cdot \bs)\xi (\sigma\cdot \bs)^*)(\bs\xi  \bs^*)}{(\bs\bs^*)^2}
  + \bigg( \frac{\bs\xi  \bs^*}{\bs\bs^*} \bigg)^2 \bigg( \frac{(\sigma\cdot\bs)(\sigma\cdot\bs)^*}{\bs\bs^*} \bigg).
\end{equation}
Integrating the first term in the right hand side in (\ref{eq:decomposition_fibernorm}), we have
$$
  \int_X\frac{(\sigma \cdot \bs)\xi\xi^* (\sigma\cdot \bs)^*}{\bs\bs^*} d\mu_{\phi_\bs}
  = \mathrm{Tr}\bigg(
  \xi^*\xi \cdot \int_X \frac{(\sigma \cdot \bs)^*(\sigma \cdot \bs)}{\bs\bs^*} d\mu_{\phi_\bs}
   \bigg).
$$
Since
$$
  D_k:=\int_X \frac{(\sigma \cdot \bs)^*(\sigma \cdot \bs)}{\bs\bs^*} d\mu_{\phi_\bs} - \mu^{\sigma_k} (\bs)
$$
is a scalar matrix that is uniformly bounded in $k$ by Lemma \ref{lem:boundedweights}
and from the assumption $\|\mu^{\sigma_k} (\bs)\|_{\mathrm{op}}< \varepsilon$, we find that
\begin{equation}\label{eq:decomposition_1stterm}
	\int_X\frac{(\sigma \cdot \bs)\xi\xi^* (\sigma\cdot \bs)^*}{\bs\bs^*} d\mu_{\phi_\bs} \ge c \|\xi\|^2
\end{equation}
for some $c>0$ independent of $k$.
Let us consider the second and third terms in (\ref{eq:decomposition_fibernorm}).
Completing the square, we have
\begin{eqnarray*}
  &&
    \bigg( \frac{\bs\xi  \bs^*}{\bs\bs^*} \bigg)^2 \bigg( \frac{(\sigma\cdot\bs)(\sigma\cdot\bs)^*}{\bs\bs^*} \bigg)
    -2 \frac{((\sigma \cdot \bs)\xi (\sigma\cdot \bs)^*)(\bs\xi  \bs^*)}{(\bs\bs^*)^2}
  \\
	&\ge&
     -
     \bigg( \frac{(\sigma\cdot\bs)(\sigma\cdot\bs)^*}{\bs\bs^*} \bigg)^{-2}
     \bigg(\frac{(\sigma \cdot \bs)\xi (\sigma\cdot \bs)^*}{\bs\bs^*}\bigg)^2
  \ge
    -c \bigg(\frac{(\sigma \cdot \bs)\xi (\sigma\cdot \bs)^*}{\bs\bs^*}\bigg)^2
\end{eqnarray*}
for some $c>0$.
In the last inequality, we use Lemma \ref{lem:boundedweights}.
Let
$$
  \varphi_\sigma(\xi):=
  \frac{(\sigma \cdot \bs)\xi (\sigma\cdot \bs)^*}{\bs\bs^*}.
$$
Following the proof of \cite[Equation (5.7)]{ps04} (here we use the R-bounded geometry of the metric), we have the Poincar\'e inequality
$$
  c \int_X (\varphi_\sigma(\xi))^2d\mu_{\phi_\bs}
  \le
  k\int_X \overline{\partial}\varphi_\sigma(\xi)\wedge \partial \varphi_\sigma(\xi) \wedge \omega_{\phi_{\bs}}^{n-1}
  + k^{-n}\bigg( \int_X \varphi_\sigma(\xi) d\mu_{\phi_\bs}\bigg)^2
$$
for some $c>0$.
On the other hand, we have
$$
   \bigg|\int_X \varphi_\sigma(\xi) d\mu_{\phi_\bs}\bigg|
   = |\mathrm{Tr}(\xi \mu^{\sigma_k} (\bs))|\le \sqrt{N_k}\|\xi\|\|\mu^{\sigma_k} (\bs)\|_{\mathrm{op}}.
$$
Therefore, using $\|\mu^{\sigma_k} (\bs)\|_{\mathrm{op}}< \varepsilon$, we have
\begin{eqnarray}
  \nonumber
		\int_X (\varphi_\sigma(\xi))^2 d\mu_{\phi_\bs}
	&\le&
    c_1 k\int_X \overline{\partial}\varphi_\sigma(\xi)\wedge \partial \varphi_\sigma(\xi) \wedge \omega_{\phi_{\bs}}^{n-1} + c_2 \|\xi\|^2
  \\
  \label{eq:decomposition_2,3term}
  &\le&
    c_3 k\|\pi_T \hat{\xi}\|^2 + c_2 \|\xi\|^2
\end{eqnarray}
for some $c_1,c_2,c_3>0$.
In the second inequality above we used that $\varphi_\sigma(\xi)$ is the Hamiltonian function of $\pi_T \hat{\xi}$ with respect to the K\"ahler metric induced by $|\cdot|_{\sigma_k\cdot h_{FS}}$, which is equivalent to $\omega_{\phi_{\bs}}$ due to Lemma \ref{lem:boundedweights} for $k\gg 0$.
Substituting (\ref{eq:decomposition_1stterm}), (\ref{eq:decomposition_2,3term}) into the integration of (\ref{eq:decomposition_fibernorm}) over $X$, we get (\ref{eq:firsteqPS}).

\noindent Proof of (\ref{eq:lasteqPS}):
The only point that is used in \cite{ps04} and fails because of the existence
of holomorphic vector fields is a $\delb$ estimate. More precisely, the following fails in general:
\begin{equation}
 \label{eq:estimdbar}
 \vert\vert w\vert\vert _{L_1^2(\om_0)} \leq C \vert \vert \delb w\vert\vert_{L^2(\om_0)}
\end{equation}
for some positive constant $C$,
with $w\in H$, where $H$ is the $L^2_1$-completion of the space of complex $T$-invariant Hamiltonian vector fields. What is true in our setting
is that (\ref{eq:estimdbar}) holds for all
$$
w\in ker(\delb_{\vert H})^{\perp_{L^2}}\simeq \Lie(T)^{\perp_{L^2}},
$$
where the orthogonal is computed with respect to the $L^2$ inner product given by integration over $X$, with the metric $\om_\bs$ on the base
and the metric $|\cdot|_{\sigma_k\cdot h_{FS}}$ on the fiber.
The argument in \cite{ps04} can be applied to our setting without modification except only one point.
In \cite{ps04}, to prove (\ref{eq:estimdbar}), the fact that the Fubini-Study metric on $T\mathbb{CP}^{N_k-1}$ has constant bisectional curvature (see (5.33) in \cite{ps04}) is used.
This does not hold in our setting.
However, Lemma \ref{lem:boundedweights} implies that the bisectional curvature is bounded uniformly.
This is sufficient to prove (\ref{eq:estimdbar}).
Hence, we can chose any lift of $\xi$ in $\Lie(G)$ to obtain (\ref{eq:lasteqPS}), because for any $t_\xi\in \Lie(T)$
\begin{equation}
 \label{eq:xitxi}
 g_{Z_k}^{\sigma_k}(\xi+t_\xi,\xi+t_\xi)=g_{Z_k}^{\sigma_k}(\xi,\xi).
\end{equation}
From (\ref{eq:firsteqPS}), (\ref{eq:secondeqPS}), (\ref{eq:lasteqPS}), we have
$$
		g_{Z_k}^{\sigma_k}(\xi,\xi)= \vert\vert \pi_N\hat{\xi} \vert\vert ^2
    \ge \frac{c_1}{k} \vert\vert \hat{\xi} \vert\vert ^2
    \ge \frac{c_2}{k^2} \vert\vert \xi \vert\vert ^2.
$$
The proof is completed.
\end{proof}

\subsection{Construction of almost $\sigma$-balanced metrics}
\label{sec:construction}
In this section we prove the following theorem to obtain the approximated $\sigma$-balanced metrics.
\begin{theorem}
 \label{theo:approxmetric}
 Let $\om_{ex}$ be a $T$-invariant extremal metric in the class $2\pi c_1(L)$, where $T\subset \Aut_0(X)$ is a maximal compact torus. Let $\sigma_k$ be the optimal weights associated
 with this torus. Then there are $T$-invariant functions $\eta_j\in C^\infty(X,\R)^T$ such that  for each $q>0$ the metrics
 $$
 \om_q(k)=\om_\infty+i \del\delb (\sum_{j=1}^q \eta_j k^{-j})
 $$
 satisfy the following:
 \begin{equation}\label{eq:approxbergmanpsi}
 k^{-n}\rho_k(\om_q(k))=\exp(\psi_k(\om_q(k))) + \cO(k^{-q-2})
 \end{equation}
\end{theorem}
\noindent First, we show the following expansion of $\exp(\psi_{k}( \omega))$ for a given $\omega$.
\begin{proposition}
\label{prop:expansionpsi}
Let $\om$ be a $T$-invariant metric.
There exist $T$-invariant functions $B_{j} (\omega)$ such that for each $q>0$
$$
\exp(\psi_{\sigma_{k}, \omega})=\sum_{j=0}^q k^{-j} B_j(\om)+ e_{q}(\omega, k)
$$
satisfying that for any $l \in \N$,
there is a constant $C_{l,q}$ such that
$$
\vert\vert e_{q}(\omega, k)  \vert \vert_{C^l}
\leq  C_{l,q}k^{-q-1}.
$$
\end{proposition}

\begin{proof}
From the proof of Proposition \ref{prop:weight},
we find that for each $q>0$
$$
v_{k}=\sum_j \nu_j k^{-j} + \cO(k^{-q-1})
$$
where $\nu_{j}$ is defined in Lemma \ref{lem:appro_vf}.
In fact, we can get the estimate (\ref{eq:estimateb}) for any power in $k$ by increasing $q$ in Lemma \ref{lem:appro_vf}.
Then for any $T$-invariant metric $\om_\phi$,
we deduce a uniform expansion in $C^{l}(X,\R)$-topology in the space $\cH$ for
$$
\sigma_k^*\om_\phi-\om_\phi=i\del\delb \psi_{\sigma_k,\phi}=i\del\delb \sum_j^{q} \theta_j k^{-j}+ \cO(k^{-q-1})
$$
as in (\ref{eq:expansionpsi}).
From this we deduce the expansion for $\exp(\psi_{\sigma_k,\phi})$.
\end{proof}
\noindent We will need the following Lemmas:
\begin{lemma}
\label{lem:expansionpsi}
 Let $\om$ be a $T$-invariant metric. Then
 $$
 	B_0(\om)=1, \,\,
	B_1(\om)=\frac{1}{2}(\theta_{ex, \om}+\underline{S})
$$
where $\theta_{ex, \om}$ is the holomorphy potential of the extremal vector field with respect to $\om$.
 Moreover, if $\om$ is extremal
 $$
4 D_\om B_1(\phi)=\nabla \phi\cdot\nabla S(\om)
 $$
\end{lemma}

\begin{proof}
The first statement follows because $\nu_{1}$ is equal to $v_{ex}$, see Lemma \ref{lem:appro_vf}.
The second statement follows from the computation of the differential of $B_1$ and is standard, see e.g. \cite[Lemma 5.2.9]{gbook}.
 \end{proof}

\begin{lemma}
\label{lem:sameexpansions}
Let $\om$ be any $T$-invariant metric. Then for any $v\in \Lie(T^c)$,
 \begin{equation}
  \label{eq:sameexpansions}
  \int_X \tilde{\theta}_{v,\om} (1+k^{-1}\Delta_{\om})e^{\psi_{\sigma_k,\om}}d\mu_{\om}= \int_X \tilde{\theta}_{v,\om} (1+k^{-1}\Delta_{\om})k^{-n}\rho_k(\om) d\mu_{\om}
 \end{equation}
 where $\tilde{\theta}_{v,\om}$ is the mean value zero holomorphy potential of $v$ with respect to $\om$.
\end{lemma}

\begin{proof}
Note that through this proof, the Laplacian considered is the complex Laplacian while in \cite{gbook} this is the $d$-Laplacian.
From the choice of the weights $\sigma_k$, we have
$$
\cF^{\sigma_k}(v)=0
$$
thus
$$
\int_X \theta_{v,\om} (1+k^{-1}\Delta_\om)e^{\psi_{\sigma_k,\om}}d\mu_\om=0
$$
for any $v\in \Lie(T^c)$ and any $T$-invariant $\om$ that is a pullback of the Fubini-Study metric.
We recall (from the proof of Proposition \ref{prop:futakimorita}) that there is a constant $c_k$ depending on $k$ such that
$\theta_{v,\om}=c_k+\tilde{\theta}_{v,\om}$, where $\tilde{\theta}_{v,\om} $ has mean value zero.
Then,
$$
 \int_X \tilde{\theta}_{v,\om} (1+k^{-1}\Delta_{\om})e^{\psi_{\sigma_k,\om}}d\mu_{\om}=-c_k\frac{N_k}{k^n}.
$$
Note that the above equation makes sense for any $T$-invariant metric (even non pulled-back metrics).
We now consider the action induced by $v$ on $V_k$
(see \cite[Proposition 8.6.1 page 200]{gbook}).
We obtain
$$
k^{-(n+1)}
w_{k}
=\int_X (1+k^{-1}\Delta_{\om})\theta_{v,\om}k^{-n}\rho_k(\om) d\mu_{\om}.
$$
As we lift the $v$ action into $SL(V_k)$, the weight vanishes and we have
\begin{equation}
 \label{eq:expansionA}
 \int_X (1+k^{-1}\Delta_{\om})\tilde{\theta}_{v,\om}k^{-n}\rho_k(\om) d\mu_{\om}=-c_k\frac{N_k}{k^n}
\end{equation}
for any $T$-invariant metric.
The result follows.
\end{proof}

\begin{proof}[Proof of Theorem \ref{theo:approxmetric}]
In the following, we only consider $T$-invariant functions. We will ommit the supscript $T$, but we shall keep
in mind that all the functions considered are supposed to be $T$-invariant. In particular, if $\L_g$ is the Lichnerowicz operator, we restrict to $\ker(\L_g)^T$,
that is to $T$-invariant Killing potentials. As $T$ is maximal, these potentials are exactly the Killing potentials of the elements of $\Lie(T)$.
The proof is then by induction on $q$.
 Write down the expansions
 $$
  k^{-n}\rho_k(\om_{ex}+i\del\delb\eta)=\sum_{j=0}^\infty A_j(\om_{ex}+\eta)k^{-j}
 $$
 and
 $$
 \exp(\psi_k(\om_{ex}+i\del\delb \eta))=\sum_{j=0}^\infty B_j(\om_{ex}+\eta)k^{-j}
 $$
 where we set
 $$
 \eta:=\sum_{l=1}^q \eta_l k^{-l}.
 $$
 We use the Taylor expansions of the coefficients $A_j$ and $B_j$ to obtain
 $$
 k^{-n}\rho_k(\om_{ex}+i\del\delb\eta)=\sum_{j=0}^\infty A_j(\om_{ex})k^{-j}+\sum_{j,l}A_{j,l}(\eta)k^{-j-l}
 $$
 and
 $$
 \exp(\psi_k(\om_{ex}+i\del\delb \eta))=\sum_{j=0}^\infty B_j(\om_{ex})k^{-j}+\sum_{j,l}B_{j,l}(\eta)k^{-j-l}
 $$
 where the $A_{j,l}(\eta)$ and $B_{j,l}(\eta)$ are polynomial expressions in the $\eta_l$ and their derivatives, depending on $\om_{ex}$.
 Assume that the $T$-invariant functions $(\eta_j)_{j\leq q-1}$ are chosen so that the above expansions agree till order $q$. We try to choose $\eta_q$
 so that the expansions agree till order $q+1$. The coefficients of order $k^{-(q+1)}$ in the two expansions are
 $$
 A_{q+1}(\om_{ex})+\sum_{I_{q+1}} A_{j,l}(\eta_1,\ldots,\eta_{q-1})+ \frac{1}{2}DS_{\om_{ex}}(\eta_q)
 $$
 and
 $$
 B_{q+1}(\om_{ex})+\sum_{I_{q+1}} B_{j,l}(\eta_1,\ldots,\eta_{q-1})+ \frac{1}{4}\nabla\eta_q\cdot\nabla S(\om_{ex})
 $$
 where we used the fact that $\om_{ex}$ is extremal together with Lemma \ref{lem:expansionpsi}. Here the sets of indices $I_{q+1}$
 are defined by the above expressions. Then the terms of order $q+1$ will agree if and only if we have
 \begin{equation}
  \label{eq:qthtermequality}
  \frac{1}{2}\L_{\om_{ex}}(\eta_q)=A_{q+1}(\om_{ex})-B_{q+1}(\om_{ex})+\sum_{I_{q+1}} (A_{j,l}-B_{j,l})(\eta_1,\ldots,\eta_{q-1})
 \end{equation}
 where $\L_g$ is the Lichnerowicz operator of any metric $g$.
The equation (\ref{eq:qthtermequality}) has a solution if and only if
\begin{equation}
 \label{eq:qthtermperp}
 A_{q+1}(\om_{ex})-B_{q+1}(\om_{ex})+\sum_{I_{q+1}} (A_{j,l}-B_{j,l})(\eta_1,\ldots,\eta_{q-1})\in \ker(\L_{\om_{ex}})^\perp.
\end{equation}
We cannot say much about (\ref{eq:qthtermperp}), but it only depends on $\eta_1,\ldots,\eta_{q-1}$
so we will add in the recursive process the asumption that at each step, (\ref{eq:qthtermperp}) is satisfied. Then equation (\ref{eq:qthtermequality})
can be solved recursively. Note that the initialization of the process requires
\begin{equation}
 \label{eq:initial}
 A_2-B_2(\om_{ex})\in \ker(\L_{\om_{ex}})^\perp.
\end{equation}
To simplify notations, set
$$
R_{q+2}(\eta_1,\cdots,\eta_q)= A_{q+2}(\om_{ex})-B_{q+2}(\om_{ex})+\sum_{I_{q+2}} (A_{j,l}-B_{j,l})(\eta_1,\ldots,\eta_{q}).
$$
It remains to show (\ref{eq:initial}) and that, when solving (\ref{eq:qthtermequality}), we can choose $\eta_q$
so that the following is true:
\begin{equation}
 \label{eq:q+1thtermperp}
 R_{q+2}(\eta_1,\cdots,\eta_q)\in \ker(\L_{\om_{ex}})^\perp.
\end{equation}
We now apply Lemma \ref{lem:sameexpansions} to $$\om_\eta:=\om_{ex}+i\del\delb \eta=\om_{ex}+i\del\delb\sum_{l=1}^q \eta_l k^{-l}.$$
Equation (\ref{eq:sameexpansions}) can be written
\begin{equation}
 \label{eq:sameexpansions2}
 \int_X \tilde{\theta}_{v,\om_\eta} (1+k^{-1}\Delta_{\om_\eta})(k^{-n}\rho_k(\om_\eta)-e^{\psi_{\sigma_k,\om_\eta}} )d\mu_{\om_\eta}=0.
\end{equation}
Then, by the induction hypothesis (choice of $\eta_1,\cdots,\eta_q$), we have the following expansion:
\begin{equation}
 \label{eq:inductionHYP}
 k^{-n}\rho_k(\om_\eta)-e^{\psi_{\sigma_k,\om_\eta}}=R_{q+2}(\eta_1,\cdots,\eta_q)k^{-(q+2)}+\cO(k^{-(q+3)})
\end{equation}
We also have:
$$
\om_\eta=\om_{ex}+\cO(k^{-1}).
$$
Thus we deduce with (\ref{eq:inductionHYP}) in equation (\ref{eq:sameexpansions2}),
that the term of order $k^{-(q+2)}$ in the expansion vanishes, that is
$$
\int_X \tilde{\theta}_{v,\om_{ex}}R_{q+2}(\eta_1,\cdots,\eta_q) d\mu_{\om_{ex}}=0.
$$
Note also that the above argument with $\eta=0$ gives (\ref{eq:initial}). The proof is complete.
\end{proof}

\subsection{Completion of Proof of Theorem \ref{thm:A}}
Once we have Proposition \ref{prop:controldmu} and Theorem \ref{theo:approxmetric}, the proof of Theorem \ref{thm:A} is almost identical to \cite{Don01}.
We give the outline of the proof.
Fix an arbitrary $R>1$.
Fix an integer $q$ determined later.
For the K\"ahler form $\omega_{q}(k)$ in Theorem \ref{theo:approxmetric}, we have
$$
		k^{-n} \rho_{k}(\omega_{q} (k))
	=
		 \exp(\psi_{k}(\omega_{q}(k)))(1 + \epsilon_{k})
$$
where $\epsilon_{k}= \cO(k^{-q-2})$.
Let
$$
	\omega'(k):= \omega_{q}(k) + i \partial\bar{\partial} \log \big(\exp(\psi_{k}(\omega_{q}(k)))(1+\epsilon_{k})\big)
	= \omega_{\bs_{0}}
$$
where $\bs_{0}$ is an orthonormal basis with respect to $\Hilb_{k}(\omega_{q}(k))$.
From Proposition 27 in \cite{Don01}, for large $k$, we find that there exists some (small) constant $c>0$ depending only on $R$
such that if $a \in \Lie(G_{k})$ satisfies
$
	\|a\|_{ \mathrm{op}}< c,
$
 then
\begin{enumerate}
\item $\bs_{0} \cdot e^{ia}$ is $R$-bounded, and
\item there exists $C_{1}$ such that
	$$
		 \|[\mu^{ \sigma_{k}}(\bs_{0}  \cdot e^{ia}) ]\|_{ \mathrm{op}}
		\le  C_{1}
		(\|a\|_{ \mathrm{op}}+\|\epsilon_{k}\|_{C^{2}, \omega_{ ex}}).
	$$
\end{enumerate}
In particular, we have
$$
	\|[\mu^{ \sigma_{k}}(\bs_{0})]\|_{ \mathrm{op}} \le C_{2} k^{-q-2}.
$$
Proposition \ref{prop:controldmu} implies that if $a$ satisfies
$$
	C_{1} (\|a\|_{ \mathrm{op}}+\|\epsilon_{k}\|_{C^{2}, \omega_{ ex}}) < \varepsilon
$$
where $\varepsilon$ is defined in Proposition \ref{prop:controldmu}, then
$$
	\Lambda_{\bs  \cdot e^{ia}} \le C_{3} k^{2}
$$
for some $C_{3}$.
Now, we will apply Proposition \ref{prop:Dmmap}
by putting
$
	Z:= Z^{T}(V_{k})
$
with $\omega_{Z}$ defined in Proposition \ref{prop:symp},
$
	G:= G_{k}
$
defined in (\ref{eq:Gk}) and $\mu:= \mu^{ \sigma_{k}}$.
Note that from $\Lie(G_{\bs_{0}})=\Lie(T_{k})$, the fact that $\cF^{\sigma_k}=0$ and Proposition \ref{prop:futakimorita}, item (\ref{prop:fm2}), we deduce
$$
	\mu^{ \sigma_{k}}(\bs_{0}) \in \Lie ((G_{k})_{\bs_{0}})^{\perp}.
$$
Let $\delta$ in Proposition \ref{prop:Dmmap} be
$$
	\min (c,\,\, \frac{\epsilon}{2C_{1}} )
$$
where  $c,\, C_{1}$ are as above. From Proposition \ref{prop:futakimorita} and $\Lie(G_{\bs_{0}})=\Lie(T_{k})$ again, we can assume that the inequality
$$
	\|[\mu^{ \sigma_{k}}(\bs_{0}  ) ]\|_{ \mathrm{ op}} \le C_{2} k^{ -q -2}
$$
still holds.
Putting $\lambda:= C_{3} k^{2}$,
$$
	\lambda \vert \vert [\mu^{ \sigma_{k}}(\bs_{0}  ) ] \vert \vert
	\le
	\sqrt{N_{k}} \lambda\|[\mu^{ \sigma_{k}}(\bs_{0}  ) ]\|_{ \mathrm{ op}}
	 < C_{2}C_{3} k^{ n/2-q }.
$$
Taking $q$ so that $n/2-q <0 $, for large $k$, we have
$$
	\lambda \vert \vert [\mu^{ \sigma_{k}}(\bs_{0}  ) ] \vert \vert < \delta.
$$
Proposition \ref{prop:Dmmap} implies that there exists $a \in \Lie ((G_{k})_{\bs_{0}})^{\perp}$ such that
$$
	\mu^{ \sigma_{k}}(\bs_{0} \cdot e^{ia})=0, \,\, \|a\| \le C_{2}C_{3} k^{ n/2-q},
$$
i.e., $\bs_{0} \cdot e^{ia}$ is $\sigma_{k}$-balanced point we desired.
By construction, considering the behavior of $C^{r}$-norm by scaling $\omega \mapsto k \omega$,
$$
	\|\omega_{\phi_{\bs_{0} \cdot e^{ia}}} - \omega_{ex}\|_{C^{r}}
	= \cO(k^{ n/2-q+r}).
$$
For any $r \ge 0$, by replacing $q$ so that $ n/2-q+r<0$, we proved that $\sigma_{k}$-balanced metrics $\omega_{\phi_{\bs_{0} \cdot e^{ia}}}$ converge to $\omega_{ex}$ in $C^{r}$-sense.
The proof of Theorem \ref{thm:A} is completed.

\subsection{Proofs of Corollaries \ref{cor:b} and \ref{cor:c}}
We sketch the proofs of Corollaries \ref{cor:b} and \ref{cor:c}, that follow from the arguments in \cite{Don01} and \cite{ah} respectively.
Let $\om$ be an extremal metric on $(X,L)$. By Theorem \ref{thm:A}, $\om$ is a limit of $\sigma_k$-balanced metrics.\\

The proof of Corollary \ref{cor:b} is as in \cite{Don01}. A $\sigma_k$-balanced metric corresponds to a zero of the moment map $\mu^{\sigma_k}$.
From general theory of moment maps, such a zero is unique, up to the $G_k$-action, in its $G_k^c$ orbit.
This fact can also be seen directly from Lemma \ref{lem:gZ}.
Assume that there exists two $\sigma_k$-balanced metrics.
Connecting them by a geodesic path $\exp(t\xi)$ in $\cB^T_{k}$, then the second derivation of $I^\sigma$ (i.e., $g_{\cB_k}^{\sigma_k}(\bs)(\xi,\xi)$) along it must be zero.
This induces that $\pi_N \hat{\xi}$ is trivial.
Hence $\exp(t\xi)$ must preserve $X$ in $\mathbb{CP}^{N_k}$, i.e., $\xi \in \Lie(\Aut_k^T)$.
This proves uniqueness of $\sigma_k$-balanced metrics.
Then the result follows at the limit.\\

The proof  of Corollary \ref{cor:c} follows the strategy from \cite{ah}.
Each $\sigma_k$-balanced metric is a product of $\sigma_k$-balanced metrics on each
factor of $(X,L^k)$. To prove the splitting for $\sigma$-balanced metrics, we use
the corresponding notion of GIT. The existence of a $\sigma$-balanced metric corresponds to the vanishing of a finite dimensional moment map,
and to a GIT stability condition. Then we use the general fact that stability for a product implies stability for each factor.
Indeed, by Hilbert-Mumford criterion, one has to check stability with respect to one parameter subgroups.
But the set of one-parameter subgroups considered for the product contains the one parameter subgroups considered for each factor.
We deduce from this that $(X_i,L_i^{\otimes k})$ admits a $\sigma_k$-balanced metric for large $k$, and by unicity, the product of these metrics
is our initial $\sigma_k$-balanced metric. Then the result follows at the limit.

\end{document}